\documentclass[11pt,a4paper]{article}
\usepackage[utf8]{inputenc}
\usepackage{amsmath}
\usepackage{amsfonts}
\usepackage[colorlinks,linkcolor=blue,anchorcolor=blue,citecolor=blue]{hyperref}
\usepackage{amssymb}
\usepackage{extarrows}
\usepackage{amsthm}
\usepackage{hyperref}
\usepackage{cite}
\usepackage{enumerate}
\usepackage[x11names]{xcolor}
\PassOptionsToPackage{normalem}{ulem}
\usepackage{ulem}
\usepackage{natbib}
\usepackage[margin=1in]{geometry} 
\usepackage{array}
\usepackage{cases}
\usepackage{mathrsfs}
\usepackage{longtable}
\usepackage{tikz}
\usetikzlibrary{shapes.geometric, arrows}
\allowdisplaybreaks[4]

\newcolumntype{L}[1]{>{\raggedright\let\newline\\\arraybackslash\hspace{0pt}}m{#1}}
\newcolumntype{C}[1]{>{\centering\let\newline\\\arraybackslash\hspace{0pt}}m{#1}}
\newcolumntype{R}[1]{>{\raggedleft\let\newline\\\arraybackslash\hspace{0pt}}m{#1}}

\usepackage{natbib}
\usepackage{color}
\usepackage{dsfont}
\usepackage{marginnote}
\usepackage{enumitem}
\usepackage{verbatim}
\usepackage{graphicx}
\theoremstyle{plain}
\newtheorem{theorem}{\protect Theorem}[section]

\newtheorem{definition}[theorem]{\protect Definition}

\newtheorem{lemma}[theorem]{\protect Lemma}
\newtheorem{remark}[theorem]{\protect Remark}
\newtheorem{ass}{\protect Assumption}
\newtheorem{corollary}[theorem]{\protect Corollary}

\def\d{\mathrm{d}}
\def\ie{\mathrm{i.e.}}
\def\as{\mathrm{a.s.}}

\newcommand{\Lip}{\mathrm{Lip}}
\newcommand{\R}{\mathbb{R}}
\newcommand{\E}{\mathbb{E}}
\newcommand{\X}{\mathcal{X}}
\newcommand{\T}{\top}
\newcommand{\F}{\mathcal{F}}
\newcommand{\Fb}{\mathbb{F}}
\newcommand{\Pb}{\mathbb{P}}
\newcommand{\Pc}{\mathcal{P}}

\newcommand{\tr}{\mathrm{tr}}
\newcommand{\Law}{\mathcal{L}}
\newcommand{\st}{\mathrm{s.t.}}

\newcommand{\Gb}{\mathbb{G}}

\newcommand{\pa}{\partial}
\newcommand{\M}{\mathbb{M}}
\newcommand{\Mc}{\mathcal{M}}

\newcommand{\U}{\mathscr{U}}
\newcommand{\Lc}{\mathcal{L}}
\newcommand{\Lb}{\mathbb{L}}
\newcommand{\lv}{\lVert}
\newcommand{\rv}{\rVert}

\newcommand{\tX}{\mathtt{X}}
\newcommand{\tY}{\mathtt{Y}}
\newcommand{\tH}{\mathtt{H}}
\newcommand{\BV}{\mathrm{BV}}

\newcommand{\assref}[1]{\hyperref[#1]{Assumption \ref*{#1}}}
\newcommand{\thmref}[1]{\hyperref[#1]{Theorem \ref*{#1}}}
\newcommand{\propref}[1]{\hyperref[#1]{Proposition \ref*{#1}}}
\newcommand{\remref}[1]{\hyperref[#1]{Remark \ref*{#1}}}
\newcommand{\lemref}[1]{\hyperref[#1]{Lemma \ref*{#1}}}
\newcommand{\defref}[1]{\hyperref[#1]{Definition \ref*{#1}}}
\newcommand{\corref}[1]{\hyperref[#1]{Corollary \ref*{#1}}}
\newcommand{\equref}[1]{\hyperref[#1]{(\ref*{#1})}}
\newcommand{\exaref}[1]{\hyperref[#1]{Example \ref*{#1}}}

\tikzstyle{startstop} = [rectangle, minimum width=3.5cm, minimum height=1cm, text centered, draw=black]
\tikzstyle{process} = [rectangle, minimum width=3.5cm, minimum height=1cm, text centered, draw=black]
\tikzstyle{arrow} = [thick, ->, >=stealth]

\author{Lijun Bo \thanks{Email: lijunbo@ustc.edu.cn, School of Mathematics and Statistics, Xidian University, Xi'an, 710126, China.}
\and
Jingfei Wang \thanks{Email:wjf2104296@mail.ustc.edu.cn, School of Mathematical Sciences, University of Science and Technology of China, Hefei, 230026, China.}
\and
Xiang Yu \thanks{Email: xiang.yu@polyu.edu.hk, Department of Applied Mathematics, The Hong Kong Polytechnic University, Kowloon, Hong Kong, China.}
}
\title{Constrained mean-field control with singular controls: Existence, stochastic maximum principle and constrained FBSDE}
\date{ }

\begin{document}
\maketitle

\vspace{-0.1in}
\begin{abstract}
	This paper studies a class of mean-field control (MFC) problems with singular controls under general dynamic state-control-law constraints. We first propose a customized relaxed control formulation to cope with the dynamic mixed constraints and establish the existence of an optimal control using  compactification argument in the proper canonical spaces to accommodate singular controls. To further characterize the optimal pair of regular and singular controls, we treat the controlled McKean-Vlasov process as an infinite-dimensional equality constraint and recast the MFC problem as an optimization problem on canonical spaces with constraints on Banach space, allowing us to derive the stochastic maximum principle (SMP) and a class of constrained BSDE using a new Lagrange multipliers method. Additionally, we investigate the uniqueness and the stability result of the solution to the constrained FBSDE associated {with} the constrained MFC with singular controls.
\ \\


	
\end{abstract}
\section{Introduction}
		
To overcome the high complexity in large stochastic system caused by the mutual influence among agents, mean-field game (MFG) problems (see \cite{Huangetal2006} and \cite{LL2007}) { have} been proposed and popularized in academic studies, which is to find a Nash equilibrium in a competitive game formulation where each agent's impact on the aggregate distribution is negligible. MFC problems {have} also gained considerable attention, which concerns the cooperative population where all agents share the same goal to attain the social optimum. The present paper targets at the MFC problem,  which is often interpreted as an optimal control of McKean-Vlasov dynamics by a social planner. Some pioneer studies on this topic, just to name a few, can be found in \cite{AD2010,Buckdahn2011,Carmona2015,Carmona1,BWWY} for the SMP method; \cite{Lacker17} and \cite{Djeteetal2022} for the limit theory; \cite{Lauriere2014} and \cite{PW17}  for the dynamic programming principle (DPP) and HJB equation on the Wasserstein space of probability measures; \cite{Bayraktaretal2018}, \cite{Wuzhang2020} and \cite{Zhou2024} for the viscosity solution and comparison principle of the HJB equation. However, it is notable that most existing studies on MFC are limited to regular controls.
		
Singular control, also called bounded variation stochastic control, allows the displacement of the states to be discontinuous. In single agent's models, optimal singular control problems have been extensively investigated. For instance, we refer to \cite{Ma92} for discussions on smooth fit principle; \cite{Haussmann} for the existence of optimal singular controls based on the relaxed control formulation; \cite{Haussmannb} for the proof of DPP to account for singular controls; \cite{Bahlali} for the SMP tailor-made for singular controls; \cite{GTom2009} for a general theory on smooth fit principle. In the competitive game framework, there are a handful of studies on singular controls, see \cite{CFR13} for the n-player game with singular control as well as \cite{FH,Campi22,CDF23} for some MFG problems with singular control. However, theoretical studies for MFC problems with singular control are underexplored in the literature, and only a recent paper \cite{DenH23} studies the derivation of DPP and the quasi-variational inequality using two-layer parameterisations of stochastic processes. 
		
The goal of the paper is to contribute to the study of MFC problems by featuring both the singular control and the joint dynamic state-control-law constraints. Motivated by practical risk regulations, liquidity requirements, or feasibility restraints, different types of constraints have been considered in stochastic control problems, but predominantly in single agent's model. Recently, in mean-field models with only regular controls, there are some emerging {research studies} to incorporate state and/or law constraints, see \cite{Germain23} for the level-set approach to tackle the law constraint in MFC problems; \cite{Daudin23a} for the analysis of mean-field limits under state constraints; \cite{Daudin23b} for the formulation of controlled FP equation under state constraints; \cite{BWY} for the optimization formulation of the constrained MFC problem and the derivation of the SMP using the generalized Lagrange multipliers method. 
		
It is evident that the allowance of singular control in the mean-field model will entangle with the dynamic mixed constraints, which renders some existing methods for constrained MFC with regular controls not applicable in our framework. For example, it is difficult to generalize the level set approach in \cite{Germain23} to account for both control constraints and the singular control. Due to the possible discontinuity of the state process and its law under the singular control, characterizing the associated FP equation is still an open problem in the literature, and hence the approach in \cite{Daudin23b} is also not suitable to handle the singular control. On the other hand, some conventional methods that have been employed in unconstrained MFG and MFC problems with singular control such as DPP and compactification arguments also need to be carefully and substantially modified in the setting when pathwise constraints are imposed jointly on the state process and its distribution flow. 
		
The present paper fills the gap in the literature by advancing some theoretical foundations for constrained MFC problems with singular control. The first contribution of the paper is to address the existence of an optimal pair of regular and singular controls by developing the customized relaxed control formulation and martingale problem to embed the dynamic mixed constraints. That is, we embed the dynamic mixed constraints in item (iv) of the admissible control rules in Definition \ref{admissible}. We then present in Lemma \ref{equivalenceY} a martingale problem characterization associated to the admissible control rules and establish in Lemma \ref{closedset} the closedness property of the set of probability measures fulfilling the constraints. By developing technical compactification arguments to accommodate the singular control and combining the previous lemmas using the embedded constraints, we are able to show that the set of approximately optimal control rules is compact (see Lemma \ref{compactness}), and hence the set of optimal control rules is compact and convex that yields the existence of an optimal relaxed control (see Theorem \ref{existence_opt}). By standard argument, we also obtain the existence of an optimal strict control for the constrained MFC problem with singular control. { We emphasize that the classical compactification technique for singular controls in \cite{Haussmann} has only been developed in unconstrained settings. In contrast, our work is, to the best of our knowledge, the first to consider compactification under joint constraints over the control, state and law. Rather than directly coping with the dynamic constraints, we transform the problem into constraints over the probability measures on the canonical space. This allows us to extend the relaxed-control framework to the constrained setting. The difficulty caused by constraints is handled by a new result established in Lemma \ref{lsc}, which is new to the literature. Moreover, by leveraging the measurable selection result established in Lemma~\ref{strictcontrol}, we can convert any probability measure on the canonical space into a controlled process that satisfies the original joint dynamic constraints.}
		
Our second contribution is to further characterize the optimal controls by establishing the SMP for the constrained MFC problem with singular control and derive the adjoint constrained BSDE. However, some conventional techniques such as the first order variation or the spike variation of controls may fail in our setting as the mixed constraints might be violated. Inspired by our recent study \cite{BWY} for constrained MFC with regular controls, we propose to regard the controlled McKean-Vlasov dynamics with singular control as an infinite-dimensional equality constraint and recast the MFC problem as an optimization problem on some proper canonical spaces with general constraints on Banach space (see the formulation \eqref{func_def}-\eqref{optimization}). Similar to \cite{BWY}, we are allowed to derive the SMP by using the generalized Fritz-John (FJ) optimality condition and the Lagrange multipliers method. To be more precise, when the forward McKean-Vlasov SDE is taken as an equality constraint, the FJ condition for the control variable gives the SMP, and the Lagrange multipliers in the space of It\^{o} processes induced by the FJ condition on the state variable yields the adjoint processes as the solution to the constrained BSDE, see Theorem \ref{SMP_con}. {To this end, the core idea is to treat the pair $(X,\alpha)\in\mathbb{M}^2\times\mathscr{U}[0,T]$ of the state and control processes as the optimization variable within an abstract constrained optimization problem on Banach spaces. A crucial step is to consider the Hilbert space $\mathbb{M}^2$ for the state process $X=(X_t)_{t\in[0,T]}$. In particular, to cope with singular controls, it is necessary to modify the space $\mathbb{M}^2$ by integrating the singular control component (refer to the space $\mathscr{K}$ defined by \eqref{eq:mathscrK}). Consequently, we must endow the newly defined policy set $\mathscr{K}$ with an appropriate topology, enabling us to utilize the Lagrangian multiplier method proposed in \cite{BWY} to derive the SMP for our constrained MFC with singular controls.}   
		
The third contribution is to analyze the non-standard constrained FBSDE in \eqref{FBSDE}-\eqref{FBSDE2} associated with our constrained MFC problem with singular control. It is worth noting that the Lagrange multiplier $\eta$ in Theorem \ref{SMP_con} is treated as part of the solution to the constrained BSDE.  {The constrained FBSDE \eqref{FBSDE}-\eqref{FBSDE2} differs substantially from the conventional reflected BSDE that has been extensively studied since \cite{ElKaroui1997R}. A typical reflected BSDE takes the form given by
\begin{align*}
Y_t &= \xi + \int_t^T f(s,Y_s,Z_s)\d s + K_T - K_t - \int_t^T Z_s\d W_s,~~Y_t\geq S_t,~~\int_0^T (Y_t - S_t)\d K_t = 0,
\end{align*}
where, $K=(K_t)_{t\in[0,T]}$ is an adapted increasing process enforcing the constraint. In contrast, in our formulation (Theorem \ref{SMP_con}), the boundary condition for the adjoint process $Y$ is characterized by $\int_0^T (G(t)^{\top} Y_t + c) \mathrm{d}\widehat{\zeta}_t = 0$, 
where the singular control $\widehat{\zeta}=(\widehat{\zeta}_t)_{t\in[0,T]}$ enters the forward dynamics of the state process rather than the backward equation. Moreover, a complementary slackness condition is imposed on the forward component $\mathbb{E}[\int_0^T\phi^i(t,\widehat{X}_t,\widehat{\mu}_t,\widehat{\alpha}_t)\eta_t^i\mathrm{d}t]=0$ for $i\in I$, where the multiplier $\eta^i=(\eta_t^i)_{t\in[0,T]}$ appears only in the backward dynamics. This cross-coupled structure leads to a new and more intricate FBSDE system.} We then verify the constraint qualification and obtain the KKT optimality condition in Lemma \ref{FBSDE_solution}. As a result, we further establish the uniqueness of the solution to the nonstandard constrained FBSDE in \eqref{FBSDE}-\eqref{FBSDE2}, see Theorem \ref{unique}. Furthermore, under some enhanced assumptions, we also prove the stability result of the solution in the sense of \eqref{stability} in Theorem \ref{stable}.
		
		The rest of the paper is organized as follows. Section \ref{sec:formulation} introduces the constrained MFC problem with singular control and standing assumptions imposed on the model. Section~\ref{sec:well-posedRelax} gives the relaxed control formulation and establishes the existence of optimal control to the constrained MFC problem by applying compactification arguments. Section \ref{sec:SMP} formulates the equivalent optimization problem with constraints in Banach space, for which we develop the SMP and {adjoint} constrained BSDE using Lagrange multipliers method. Lastly, the uniqueness and the stability result of the solution to the constrained FBSDE associated to the constrained MFC problem with singular control are discussed in Section \ref{sec:uniqueFBSDE}. The proofs of some auxiliary results are provided in \hyperref[appn]{Appendix}.
		
{\small {\bf Notations.} We list below some notations that will be frequently used throughout the paper:
\vspace{-0.2in}
\begin{center}
\begin{longtable}{ll}
$(E^*, \lv\cdot\rv_{E^*})$ & Dual space of a Banach space $(E, \lv\cdot\rv_E)$ \\ 
$\langle\cdot,\cdot\rangle_{E,E^*}$ & Pairing between Banach space $(E,\lv\cdot\rv_E)$ and its dual \\ 
$a \geq_K 0~(a >_K 0)$ & $a \in K~(a \in K^{\mathrm{o}})$, with $K$ a closed (linear) cone \\ 
$Df$ & Fr\'echet derivative of $f: E \to \R$ \\ 
$\nabla_x f$ & Gradient of $f: \R^n \to \R$ \\ 
$D_X f$ & Partial Fr\'echet derivative of $f$ with respect to $X$ \\ 
$D_{\alpha} f$ & Partial Fr\'echet derivative of $f$ with respect to $\alpha$ \\
$L^p((A,\mathscr{B}(A),\lambda_A);E)$ & Set of $L^p$-integrable $E$-valued mappings on measure space \\
& $(A,\mathscr{B}(A),\lambda_A)$; write $L^p(A;E)$ for short \\ 
$L^p((A,\mathscr{B}(A),\lambda_A);E)/\mathord{\sim}$ & Quotient set of $L^p(A;E)$ under relation $f\mathord{\sim}g$ iff $f=g$ $\lambda_A$-a.e.; \\
& write $L^p(A;E)/\mathord{\sim}$ for short \\ 
$\E$ ($\E'$) & Expectation operator under probability measure $\Pb$ ($\Pb'$) \\ 
${\cal P}_p(E)$ & Set of probability measures on $E$ with finite $p$-order moments \\
{${\cal W}_{p,E}(\mu,\nu)$} & { $p$-order Wasserstein metric between $\mu\in{\cal P}_p(E)$ and $\nu\in{\cal P}_p(E)$}\\
{$\partial_{\mu}f$} & {Lions derivative of $f:{\cal P}_2(\R^n)\to\R$ with respect to $\mu\in{\cal P}_2(\R^n)$}\\
$\Mc[0,T]$ ($\Mc^+[0,T]$) & Set of signed (non-negative) Radon measures on $[0,T]$ \\ 
$\BV^k[0,T]$ & Set of functions $f:[0,T]\to\R^k$ with bounded variation \\ 
$D^{k}[0,T]$ & Set of c\`adl\`ag functions $f:[0,T]\to\R^k$ \\
\end{longtable}
\end{center}}

\section{Problem Formulation}\label{sec:formulation}
\subsection{MFC with singular control and mixed constraints}
		
Let $(\Omega,\F,\Fb,\Pb)$ be a filtered  probability space with the filtration $\Fb=(\F_t)_{t\in[0,T]}$ satisfying the usual conditions. For $r,l,k\in\mathbb{N}$ and $p>2$, let $W=(W_t)_{t\in[0,T]}$ be a standard $r$-dimensional $(\Pb,\Fb)$-Brownian motion and $U\subset\R^l$ be a closed and convex (regular) set as the policy space. Define the space of regular controls $\U[0,T]$ as the set of $\Fb$-progressively measurable processes such that $\E[\int_0^T|\alpha_t|^p\d t]<\infty$.  For the space of singular controls, define $A^k[0,T]$ as the set of functions $a:[0,T]\mapsto\R^k$ such that $a=(a_t^i)_{t\in[0,T]}$ is c\`adl\`ag and non-decreasing with $a_0^i\geq 0$ for $i=1,\dots,k$. 
Let $\mathcal{A}^k[0,T]$ be the set of $\Fb$-progressively measurable processes $\zeta=(\zeta_t)_{t\in[0,T]}$ with paths in $A^k[0,T]$ and $\E[|\zeta_T|^p]<\infty$.
		
Let $(b,\sigma):[0,T]\times\R^n\times\Pc_2(\R^n)\times\R^l\mapsto\R^n\times\R^{n\times r}$ be measurable mappings and $G:[0,T]\mapsto\R^{n\times k}$ be a continuous function. For $(\alpha,\zeta)=(\alpha_t,\zeta_t)_{t\in[0,T]}\in\U[0,T]\times\mathcal{A}^k[0,T]$, we consider the controlled McKean-Vlasov SDE that, for $t\in(0,T]$,
\begin{align}\label{eq:controlledSDE}
\d X_t^{\alpha,\zeta} &=b(t,X_t^{\alpha,\zeta},\Law(X_t^{\alpha,\zeta}),\alpha_t)\d t +\sigma(t,X_t^{\alpha,\zeta},\Law(X_t^{\alpha,\zeta}),\alpha_t)\d W_t+G(t)\d\zeta_t,
\end{align}
and $X_0^{\alpha,\zeta}=\xi\in L^p((\Omega,\F_0,\Pb);\R^n)$ with the law of $\xi$ under $\mathbb{P}$ being $\Law(\xi)=\nu\in\Pc_p(\R^n)$. 
		
For the above controlled McKean-Vlasov dynamics, we introduce dynamic state-control-law constraints. Let $d\in\mathbb{N}$ be the number of inequality constraints. For $i\in I:=\{1,\dots,d\}$, consider $\phi^i:[0,T]\times \R^n\times\Pc_2(\R^n)\times\R^l\mapsto\R$ with $i\in I$ such that the dynamic state-control-law inequality constraints are described by, $\d t\times\d\Pb$-a.s.
\begin{align}\label{path_constraint}
\phi^i(t,X_t^{\alpha,\zeta},\Law(X_t^{\alpha,\zeta}),\alpha_t)\geq 0,
\end{align}
		where $X^{\alpha,\zeta}=(X_t^{\alpha,\zeta})_{t\in[0,T]}$ obeys the McKean-Vlasov SDE \eqref{eq:controlledSDE}. Denote by $\U_{\rm ad}[0,T]$ the admissible control set of processes $(\alpha,\zeta)\in\U[0,T]\times\mathcal{A}^k[0,T]$ such that there exists a c\`{a}dl\`{a}g $\Fb$-adapted ($\R^n$-valued) process $X^{\alpha,\zeta}=(X_t^{\alpha,\zeta})_{t\in[0,T]}$ satisfying \eqref{eq:controlledSDE} and \eqref{path_constraint}.
		
The goal of the MFC problem is to minimize the following cost functional that, for $(\alpha,\zeta)=(\alpha_t,\zeta_t)_{t\in [0,T]}\in\U_{\rm ad}[0,T]\neq\varnothing$, 
\begin{align}\label{cost_func0}
&\inf_{(\alpha,\zeta)\in\U_{\rm ad}[0,T] }J(\alpha,\zeta)\\
&\qquad:=\inf_{(\alpha,\zeta)\in\U_{\rm ad}[0,T] }\E\left[\int_0^Tf(t,X_t^{\alpha,\zeta},\Law(X_t^{\alpha,\zeta}),\alpha_t)\d t+{\int_{[0,T]} c(t)d\zeta_t+g(X_T^{\alpha,\zeta},\Law(X_T^{\alpha,\zeta})}\right],\nonumber
\end{align}
where $f:[0,T]\times\R^n\times\Pc_2(\R^n)\times\R^l\mapsto\R$, $c:[0,T]\mapsto\R^k$ and $g:\R^n\times\Pc_2(\R^n)\to\R$ are respectively the running cost function, the cost function associated with singular controls and the terminal cost function. 
		
\subsection{Standing assumptions} 
Let us first unpack some standing assumptions imposed on the model coefficients, which will be used in different sections.
\begin{ass}\label{ass1} 
\begin{itemize}
\item[{\rm(A1)}] The mappings $b,\sigma,f,g$ and $\phi^i$, for $i\in I$ are Borel measurable, $b,\sigma$ are continuous and $f,g$ is lower semi-continuous (l.s.c.) in $(x,\mu,u)\in\R^n\times\Pc_2(\R^n)\times \R^l$.

\item[{\rm(A2)}] For $i\in I$, $\phi^i$ is upper semi-continuous (u.s.c.) in $(x,\mu,u)\in\R^n\times\Pc_2(\R^n)\times\R^l$, and there exists $M>0$ such that $|\phi^i(t,x,\mu,u)|\leq M(1+|x|+\sqrt{M_2(\mu)}+|u|)$ for $(i,x,u)\in I\times\R^n\times\R^l$ and {$M_2(\mu):=\int_{\R^n}|x|^2\mu(\d x)$}. Moreover, $\phi^i$ is uniformly u.s.c. in $\mu\in \Pc_2(\R^n)$ w.r.t. $(t,x,u)\in[0,T]\times\R^n\times\R^l$.

\item[{\rm(A3)}] The mappings $b(t,x,\mu,u)$ and $\sigma(t,x,\mu,u)$ are uniformly Lipschitz continuous in $(x,\mu)\in\R^n\times\Pc_2(\R^n)$, i.e., there is $M>0$ independent of $u\in \R^l$ such that, for all $(x,\mu),(x',\mu')\in\R^n\times\Pc_2(\R^n)$ and $(t,u)\in[0,T]\times \R^l$,
\begin{align*}
\left|b(t,x',\mu',u)-b(t,x,\mu,u)\right|+\left|\sigma(t,x',\mu',u)-\sigma(t,x,\mu,u)\right|\leq M(|x-x'|+\mathcal{W}_{2,\R^n}(\mu,\mu')).
\end{align*}

\item[\rm(A4)] There is $M>0$ such that $|b(t,x,\mu,u)|+|\sigma(t,x,\mu,u)|\leq M(1+|x|+\sqrt{M_2(\mu)}+|u|)$ for $(t,x,\mu,u)\in[0,T]\times\R^n\times\Pc_2(\R^n)\times \R^l$.

\item[\rm(A5)] $G:[0,T]\mapsto\R^{n\times k}$ and $c:[0,T]\mapsto [0,\infty)^k$ are deterministic continuous functions. Moreover, there is $c_0>0$ such that $c^i(t)\geq c_0$ for all $t\in [0,T]$. 

\item[\rm(A6)] There exist $M,M_1>0$ such that, for all $(t,x,\mu,u)\in [0,T]\times\R^n\times\Pc_2(\R^n) \times\R^l$, 
{\begin{align*}
-M(1+|x|^p+M_2(\mu))+M_1|u|^p&\leq f(t,x,\mu,u)\leq M(1+|x|^p+M_2(\mu)+|u|^p),  \\
 -M(1+|x|^p+M_2(\mu))&\leq g(x,\mu)\leq-M(1+|x|^p+M_2(\mu))
\end{align*}}
{where $p>2$ is given at the beginning of this section.}
\end{itemize}
\end{ass}

\begin{ass}\label{ass3} 
			\begin{itemize}
				\item[{\rm(B1)}] The mappings $b,\sigma,f,g$ and $\phi^i$  are Borel measurable and continuously differentiable w.r.t. $x\in\R^n$.  Moreover, $\phi^i$ is jointly continuous in $(t,x,\mu)\in[0,T]\times\R^n\times\Pc_2(\R^n)$.
				
				\item[{\rm(B2)}] There exists $M>0$ such that, for all $(t,x,\mu,u)\in[0,T]\times\R^n\times\Pc_2(\R^n)\times\R^l$,
				\begin{align*}
					\sum_{i\in I}\left(|\nabla_x\phi^i(t,x,\mu,u)|+|\nabla_u\phi^i(t,x,\mu,u)|\right)&\leq M.
				\end{align*}
				\item[{\rm(B3)}] The condition {\rm(A3)} in Assumption~\ref{ass1} holds.
				\item[\rm(B4)] There exists $M>0$ such that $|\delta(t,x,\mu,u)|\leq M(1+|x|+|u|+\sqrt{M_2(\mu)})$ for $(t,x,\mu,u)\in[0,T]\times\R^n\times\Pc_2(\R^n)\times \R^l$ and $\delta\in\{b,\sigma,\nabla_ub(\cdot),\nabla_u\sigma(\cdot),\nabla_xf(\cdot),\nabla_uf(\cdot),\nabla_xg\}$.
				
				\item[\rm(B5)] The mappings $b(t,x,\mu,u),\sigma(t,x,\mu,u),f(t,x,\mu,u), g(x,\mu)$ and $\phi^i(t,x,\mu,u)$ are continuously Lions-differentiable w.r.t. $\mu\in\Pc_2(\R^n)$. Moreover, there exists $M>0$ such that, for all $(t,x,\mu,u)\in[0,T]\times\R^n\times\Pc_2(\R^n)\times \R^l$,
				\begin{align*}
					 \left(\int_{\R^n}|\pa_{\mu}f(t,x,\mu,u)(x')|^2\mu(\d x')\right)^{\frac12}&\leq M\left(1+|x|+|u|+\sqrt{M_2(\mu)}\right),\\
                    \left(\int_{\R^n}|\pa_{\mu}g(x,\mu)(x')|^2\mu(\d x')\right)^{\frac12}&\leq M\left(1+|x|+\sqrt{M_2(\mu)}\right)
				\end{align*} 
				and moreover $\sum_{i\in I}\left(\int_{\R^n\times\R^l}|\pa_{\mu}\phi^i(t,x,\mu,u)(x')|^2\mu(\d x')\right)^{\frac12}\leq M$.
\item[\rm(B6)] $G:[0,T]\mapsto\R^{n\times k}$ and $c:[0,T]\mapsto [0,\infty)^k$ are deterministic continuous functions. Moreover, there exists a constant $c>0$ such that $|G(t)-G(s)|\leq c|t-s|$ and $|G(t)\boldsymbol{\beta}|^2\geq c|\boldsymbol{\beta}|$ for any vector $\boldsymbol{\beta}\in\R^k$ and $t,s\in [0,T]$.
			\end{itemize}
		\end{ass}
		
\begin{ass}\label{ass4}
\begin{itemize}
\item [\rm (C1)] The constrained function $\phi$ is independent of the control variable and has the linear form $\phi(t,x,\mu)=A(t)x+B(t)\int_{\R^n}x\mu(\d x)+C(t)$ for $(t,x,\mu)\in[0,T]\times\R^n\times\Pc_2(\R^n)$. Here, $(A,B)=(A_t,B_t)_{t\in [0,T]}\in C([0,T];\R^n\times\R^n)$ and $C=(C(t))_{t\in[0,T]}\in C([0,T];\R^n)$ and $|A(t)|>0$ and $|A(t)+B(t)|>0$ for $t\in[0,T]$. 	
\item [\rm (C2)] The Hamiltonian $H$ defined by, $H(t,x,\mu,u,y,z):=b(t,x,\mu,u)^{\T}y+\tr(\sigma(t,x,\mu,u)^{\T}z)+f(t,x,\mu,u)$ is strictly convex in $u\in U$. There exists a measurable mapping $\hat{u}:[0,T]\times\R^n\times\Pc_2(\R^n)\times\R^n\times\R^{n\times r}\mapsto U\subset\R^l$ with at most quadratic growth in $(x,\mu,y,z)\in\R^n\times\Pc_2(\R^n)\times\R^n\times\R^{n\times r}$ such that $\hat{u}(t,x,\mu,y,z)\in\mathop{\arg\min}_{u\in U} H(t,x,\mu,u,y,z)$ and $\hat u$ is Lipschitz continuous, i.e., there is $M>0$ such that, for $(x_i,\mu_i,y_i,z_i)\in\R^n\times\Pc_2(\R^n)\times\R^n\times\R^n\times\R^{n\times r}$ with $i=1,2$, 
\begin{align*}
&|\hat u(t,x_1,\mu_1,y_1,z_1)-\hat u(t,x_2,\mu_2,y_2,z_2)|\nonumber\\
&\qquad\qquad\leq M(|x_1-x_2|+\mathcal{W}_{2,\R^n}(\mu_1,\mu_2)+|y_1-y_2|+|z_1-z_2|).    
\end{align*}
Set $\hat{b}(t,x,\mu,y,z):=b(t,x,\mu,\hat{u}(t,x,\mu,y,z))$,  $\hat{\sigma}(t,x,\mu,y,z):=\sigma(t,x,\mu,\hat{u}(t,x,\mu,y,z))$ and $\hat{f}(t,x,\mu,y,z):=f(t,x,\mu,\hat{u}(t,x,\mu,y,z))$.
				
\item [\rm (C3)] Define $F(t,x,\mu,u,y,z):=\nabla_x H(t,x,\mu,u,y,z)+\int_{\R^n}\pa_{\mu}H(t,x',\mu,u,y,z)(x)\mu(\d x')$, $\hat{F}(t,x,\mu,y,z):=F(t,x,\mu,\hat{u}(t,x,\mu,y,z),y,z)$ and $\Phi(t,x,\mu,y,z):=(-\hat{F},\hat{b},\hat{\sigma})(t,x,\mu,y,z)$. The following monotone condition holds:
\begin{align*}
& (\Phi(t,x_1,\mu,y_1,z_1)-\Phi(t,x_2,\mu,y_2,z_2))(x_1-x_2,y_1-y_2,z_1-z_2)\nonumber\\
&\qquad+\beta(|x_1-x_2|^2+|y_1-y_2|^2+|z_1-z_2|^2)\leq 0,   
\end{align*}
where $\beta>\frac{M}{2}$ with $M$ being given in {\rm(A3)}. The inner product is read as
\begin{align*}
&(\Phi(t,x_1,\mu,y_1,z_1)-\Phi(t,x_2,\mu,y_2,z_2))(x_1-x_2,y_1-y_2,z_1-z_2)\\
&\qquad=-(\hat{F}(t,x_1,\mu,y_1,z_1)-\hat{F}(t,x_2,\mu,y_2,z_2)^{\T}(x_1-x_2)\\
&\qquad\quad+(\hat{b}(t,x_1,\mu,y_1,z_1)-\hat{b}(t,x_2,\mu,y_2,z_2))^{\T}(y_1-y_2)\\
&\qquad\quad+\tr\left((\hat{\sigma}(t,x_1,\mu,y_1,z_1)-\hat{\sigma}(t,x_2,\mu,y_2,z_2))^{\T}(z_1-z_2)\right).
\end{align*}
\item [\rm (C4)] $\hat{b}$, $\hat{\sigma}$, $\hat{F}$ are Lipschitz continuous, i.e., there is $M>0$ such that, for  $(x_i,\mu_i,y_i,z_i)\in\R^n\times\Pc_2(\R^n)\times\R^n\times\R^n\times\R^{n\times r}$ with $i=1,2$, $\delta\in\{\hat{b},\hat{\sigma},\hat{F}\}$ and $x\in\R^n$, 
\begin{align*}
&|\delta(t,x_1,\mu_1,y_1,z_1)-\delta(t,x_2,\mu_2,y_2,z_2)|\nonumber\\
&\qquad\leq M(|x_1-x_2|+\mathcal{W}_{2,\R^n}(\mu_1,\mu_2)+|y_1-y_2|+|z_1-z_2|).  
\end{align*}
{\item [{\rm (C5)}] Set $\Psi(x,\mu):=\nabla_x g(x,\mu)+\int_{\R^n}\pa_{\mu} g(x',\mu)(x)\mu(\d x')$ for $(x,\mu)\in\R^n\times\Pc_2(\R^n)$ and $\Psi$ satisfies the monotone condition $(\Psi(x_1,\mu)-\Psi(x_2,\mu))(x_1-x_2)\geq 0$ and $\int_{\R^n}(\Psi(x,\mu_1)-\Psi(x,\mu_2))x'(\mu_1-\mu_2)(\d x')\geq 0$ for $x,x_1,x_2\in\R^n$ and $\mu,\mu_1,\mu_2\in\Pc_2(\R^n)$.}
\end{itemize}
\end{ass}

Assumption~\ref{ass3} strengthens Assumption~\ref{ass1} by requiring the differentiability of the coefficients and imposing growth conditions on their derivatives, while it is weaker than Assumption~\ref{ass1} by relaxing the growth condition on the running cost function $f$ and omitting the requirement for a uniform positive constant $c_0$ of the cost function $c$. Assumption~\ref{ass4} introduces the monotonicity to the coefficients and guarantees the existence of a Lipschitz continuous minimizer for the Hamiltonian. 
    
\section{Existence of Optimal Controls}\label{sec:well-posedRelax}
		
\subsection{Relaxed control formulation}
To confirm the existence of an optimal pair of regular and singular controls to problem \eqref{cost_func0}, we first reformulate the problem in a relaxed control framework such that the compactification method can be utilized. To this end, let us introduce the following basic spaces:
\begin{itemize}
\item 	The space $D^n[0,T]$ is endowed with the Skorokhod metric $d_{D}$ and the corresponding Borel $\sigma$-algebra denoted by $\F^D$. Let $\F_t^D$ be the { naturally generated} Borel $\sigma$-algebra  up to time $t\in[0,T]$ (c.f. \cite{Haussmann}).

\item 	The space $\mathcal{Q}$ of relaxed controls is defined as the set of measures $q$ in $[0,T]\times U$ with the first marginal equal to the Lebesgue measure and $\int_{[0,T]\times U}|u|^pq(\d t,\d u)<\infty$. We endow the space $\mathcal{Q}$ with the $2$-Wasserstein metric on $\Pc_2([0,T]\times U)$ given by $d_{\mathcal{Q}}(q^1,q^2)=\mathcal{W}_{2,[0,T]\times U}\left(\frac{q^1}{T},\frac{q^2}{T}\right)$, where the metric on $[0,T]\times U$ is given by $((t_1,u_1),(t_2,u_2))\mapsto|t_2-t_1|+|u_2-u_1|$. Note that each $q\in\mathcal{Q}$ can be identified with a measurable function $[0,T]\in t\mapsto q_t\in\Pc_2(U)$, defined uniquely up to $\as$ by $q(\d t,\d u)=q_t(\d u)\d t$. In the sequel, we will always refer to the measurable mapping $q=(q_t)_{t\in [0,T]}$ to a relaxed control in $\mathcal{Q}$. Let $\F^{\mathcal{Q}}$ be the Borel $\sigma$-algebra of $\mathcal{Q}$ and $\F_t^{\mathcal{Q}}$ be the $\sigma$-algebra generated by the maps $q\mapsto q([0,s]\times V)$ with $s\in [0,t]$ and Borel measurable $V\subset U$. { By the same argument of Lemma 3.2 in \cite{Lacker}, one can verify that $q=(q_t)_{t\in [0,T]}$ is $\mathbb{F}^{\mathcal{ Q}}=(\F_t^{\mathcal{Q}})_{t\in[0,T]}$-progressively measurable.} 

\item The space $A^k[0,T]$ is a subset of $\mathcal{V}\subset\BV^k[0,T]$, where the set $\mathcal{V}\subset\BV^k[0,T]$ is the collection of mappings $a:[0,T]\to\R^k$ such that $a$ is c\`adl\`ag {and $a(0)\geq 0$.} As in \cite{Haussmann} {and Section 3.5 of \cite{Folland}},  there is a natural one-to-one correspondence between $\mathcal{V}$ and $(\Mc[0,T])^k$, that is, for $a=(a^1,\ldots,a^k)\in\mathcal{V}$ and $i=1,\ldots,k$, define $\nu_i(\{0\})=a_0^i$ and $\nu_i((s,t])=a^i(t)-a^i(s)$ for $0\leq s\leq t\leq T$. Then, $\mathcal{V}$ is the dual space of $(C[0,T])^k$, and hence we endow $\mathcal{V}$ with the corresponding weak${}^*$ topology (which is also equivalent to the Fortet-Mourier norm on ($\Mc[0,T])^k$ {introduced in \eqref{FM_0Tnorm} below, or c.f. \cite{BWWY}}). Thus, $A^k[0,T]$ is a closed subset of $(\Mc[0,T])^k$ {by Portmaneau theorem ($\nu^i(\{0\})\geq\limsup_{n\to\infty}\nu_i^n(\{0\})\geq 0$ for any $i=1,\dots,k$)} and can also be viewed as $(\Mc^+[0,T])^k$. Let $\F^{A}$ be the Borel $\sigma$-algebra and $\F^{A}_t$ be the { naturally generated} Borel $\sigma$-algebra defined up to time $t\in[0,T]$ (cf. \cite{Haussmann}).
\end{itemize}
		
Now, we introduce the sample space considered in the paper, which is given by  
\begin{align}\label{eq:samplespace}
\Omega:=\widehat{\Omega}\times A^k[0,T],\quad {\widehat{\Omega}= D^n[0,T]\times \mathcal{Q}.}
\end{align}
Equip $\widehat{\Omega}$ and $\Omega$ with the respective product metrics $d_{\widehat{\Omega}}(\widehat{\omega}^1,\widehat{\omega}^2)=d_{D}(\psi^1,\psi^2)+d_{\mathcal{Q}}(q^1,q^2)$ and $d_{\Omega}(\omega^1,\omega^2)=d_{\widehat{\Omega}}(\widehat{\omega}^1,\widehat{\omega}^2)+\sum_{i=1}^kd_{\rm FM}(\nu_i^1,\nu_i^2)$ for $\omega^i=(\widehat{\omega}^i,\zeta^i)=(\psi^i,q^i,\zeta^i)\in\Omega$ with $i=1,2$. Here, $\nu^i$ is the correspondence of $\zeta^i$ in $\mathcal{V}$ and $d_{\rm FM}$ is {the Fortet-Mourier metric on $\mathcal{M}[0,T]$ that, for any $\mu,\nu\in{\cal M}[0,T]$,
\begin{align}\label{FM_0Tnorm}
d_{\rm FM}(\mu,\nu) := \sup_{\substack{\|f\|_\infty+{\rm Lip}(f)\le 1}} \left|\int_{[0,T]} f\,d\mu-\int_{[0,T]} f\,d\nu\right|,     
\end{align}
which induces the weak convergence of finite measures. We adopt this metric as it is compatible with weak convergence while preserving completeness and separability, ensuring that the resulting product spaces remain Polish.}
		
Set $\F:=\F^{D}\otimes\F^{\mathcal{Q}}\otimes\F^{A}$ as the product $\sigma$-algebra and define the filtration $\Fb=(\F_t)_{t\in[0,T]}$ by $\F_t=\F_t^{D}\otimes\F_t^{\mathcal{Q}}\otimes\F_t^{A}$ for $t\in[0,T]$. For any $\omega=(\psi,q,\zeta)\in\Omega$, let $X_t(\omega)=\psi_t$, $q_t(\omega)=q_t$ and $\zeta_t(\omega)=\zeta_t$ for $t\in[0,T]$ be the coordinate process. Denote by $\Pc_2(\Omega;\widehat{\Omega})$ the set of probability measures $R$ on $(\Omega,\F)$ such that $R|_{\widehat{\Omega}}\in\Pc_2(\widehat{\Omega})$, {where $R|_{\widehat{\Omega}}$ denotes the marginal distribution of $R$ to $\widehat{\Omega}$}. Introduce the corresponding metric on $\Pc_2(\Omega;\widehat{\Omega})$ by, for any $R_1,R_2\in\Pc_2(\Omega;\widehat{\Omega})$,
\begin{align*}
d(R_1,R_2)=\mathcal{W}_{2,\widehat{\Omega}}(R_1|_{\widehat{\Omega}},R_2|_{\widehat{\Omega}})+d_{\rm FM,\Omega}(R_1,R_2),   
\end{align*}
{where, we recall that $\mathcal{W}_{2,\widehat{\Omega}}$ denotes the $2$-order Wasserstein distance on $\Pc_2(\widehat{\Omega})$, and $d_{\rm FM,\Omega}$ is the Fortet--Mourier metric on $\Pc(\Omega)$, defined by, for any $\mu,\nu\in\Pc(\Omega)$,
\begin{align}
d_{\rm FM,\Omega}(\mu,\nu):=\sup_{\substack{\|f\|_\infty\le1+{\rm Lip}(f)\le1}}\left|\int_\Omega f\,d\mu-\int_\Omega f\,d\nu\right|.
\end{align}} 
As a result, $R_n\to R$ in $\Pc_2(\Omega;\widehat{\Omega})$ iff $R_n|_{\widehat{\Omega}}\to R|_{\widehat{\Omega}}$ in $\Pc_2(\widehat{\Omega})$ and $R_n$ weakly converges to $R$ as $n\to\infty$.
		
Next, we give the definition of relaxed controls. In particular, one important task is to guarantee that the dynamic state-control-law constraints are fulfilled when applying the compactification argument. For this purpose, we propose to embed the constraints into the definition of admissible control rules in the following sense. Recall that $\nu\in\Pc_p(\R^n)$ is the law of initial state in SDE \eqref{eq:controlledSDE}. 
		
In the sequel, we will always consider the weak formulation, $\ie$, we will take $(\Omega,\F,\Fb,\Pb)$ introduced in this subsection to be the canonical space.
\begin{definition}\label{admissible}
An admissible control rule is a probability measure $R\in\Pc_2(\Omega;\widehat{\Omega})$ on $(\Omega,\F)$ such that
\begin{itemize}
\item[\rm (i)] $R\circ X_0^{-1}=\nu$; 
\item[\rm(ii)] $\E^R\left[\int_0^T\int_U|u|^pq_t(\d u)\d t+|\zeta_T|^p\right]<\infty$; 
\item[\rm (iii)] for all $\psi\in C_b^2(\R^n)$, $M^{R}\psi=(M_t^{R}\psi)_{t\in[0,T]}$ is an $(R,\Fb)$-martingale, where 
\begin{align*}
M_t^R\psi&:=\psi(X_t)-\int_0^t\int_U\mathbb{G} \psi(s,X_s,R^X(s),u)q_s(\d u)\d s-\int_0^t\nabla_x\psi(X_{s-})G(s)\d\zeta_s\\
&\quad-\sum_{0<s\leq t}\left(\psi(X_s)-\psi(X_{s-})-\nabla_x\psi(X_{s-})\Delta X_s\right)
\end{align*}
with $R^X(s):=R\circ X_s^{-1}$, $\Delta X_s:=X_s-X_{s-}$ and 
\begin{align*}			\mathbb{G}\psi(t,x,\mu,u):=b(t,x,\mu,u)\nabla_x\psi(x)+\frac12\tr\left(\sigma\sigma^{\T}(t,x,\mu,u)\nabla_{xx}^2\psi(x)\right).
\end{align*}
\item[\rm (iv)] for $i\in I$, $(m\times R)(A_i^R)=T$ with the admissible set $A_i^R$ defined by
\begin{align}\label{eq:AiR}
A_i^R=\left\{(t,\omega)\in[0,T]\times\Omega;~ \int_U\phi^i(t,X_t(\omega),R^X(t),u)q_t(\omega,\d u)\geq 0\right\},
\end{align}
where $m\otimes R$ denotes the product measure of {Lebesgue} measure $m$ and $R$.
\end{itemize}
\end{definition}

\begin{remark}
    To the best of our knowledge, the formulation of the relaxed constraint condition in Definition~\ref{admissible}-(iv) is novel within the compactification approach. Unlike the existing work, we incorporate the original state-control-law constraint directly into the definition of admissible probability measures on the canonical space. Such definition, to the best of our knowledge, has not been considered before in the compactification literature.
    \end{remark}

		\subsection{Existence of optimal relaxed controls}
		
		For $\nu\in\Pc_p(\R^n)$, let $\mathcal{R}_{\rm ad}(\nu)$ be the set of admissible control rules $R$ in Definition~\ref{admissible}. For $R\in\mathcal{R}_{\rm ad}(\nu)$, we have the following {martingale} measure characterization:
\begin{lemma}\label{moment_p}
$R\in\mathcal{R}_{\rm ad}(\nu)$ iff there is a filtered probability space $(\Omega',\F',\Fb'=(\F'_t)_{t\in[0,T]},P')$ supporting an $n$-dimensional $\Fb'$-adapted process $X=(X_t)_{t\in[0,T]}$, $r$-dimensional orthogonal $\Fb'$-martingale measures $N=(N^1,\ldots,N^r)$ on $U\times[0,T]$, each with $\F'_t$-progressively measurable intensity $q_t(\d u)\d t$ and a $k$-dimensional $\Fb'$-progressively measurable non-decreasing process $\zeta=(\zeta_t)_{t\in[0,T]}$ such that $R=P'\circ (X,q,\zeta)^{-1}$, and the following conditions hold:
			\begin{itemize}
				\item[{\rm(i)}] $P'\circ X_0^{-1}=\nu$; 
                \item[{\rm(ii)}] $\E^{P'}\left[\int_0^T\int_U|u|^pq_t(\d u)\d t+|\zeta_T|^p\right]<\infty$;
				
				\item[{\rm(iii)}] the state dynamics obeys that, $P'$-a.s.
				\begin{align*}
					\d X_t=\int_Ub(t,X_t,\Law^{P'}(X_t),u)q_t(\d u)\d t+\int_U\sigma(t,X_t,\Law^{P'}(X_t),u)N(\d u,\d t)+G(t)\d\zeta_t,
				\end{align*}
				where $\mathcal{L}^{P'}(X_t)$ denotes the law of $X_t$ under $P'$.            
				\item[{\rm(iv)}] the constraints hold that, for $m\otimes P'$-$\as$, $\int_U\phi^i(t,X_t,\Law^{P'}(X_t),u)q_t(d u)\geq 0$ for $i\in I$.
			\end{itemize}
			Moreover, there exists $C>0$ depending on $M,\nu,p,T$ such that $\E^{P'}[\sup_{t\in[0,T]}|X_t|^p]\leq C\left\{1+\E\left[|\xi|^p\right]\right\}$ with $M$ stated in Assumption~\ref{ass1}.
		\end{lemma}
		
		\begin{proof}
			The first assertion follows from {\cite{Karoui}}, and the second assertion can be easily proved by basic moment estimations under Assumption~\ref{ass1}.
		\end{proof}
		
		Let us define the cost functional on $\mathcal{R}_{\rm ad}(\nu)$ by
		\begin{align}\label{eq:objectGammaR}
			\Gamma(R):=\E^R\left[\int_0^T\int_U f(t,X_t,R^X(t),u)q_t(\d u)\d t+{\int_{[0,T]} c(t)d\zeta_t+g(X_T,R^X(T))}\right]. 
		\end{align}
		By using Assumption~\ref{ass1}, $v^*:=\inf_{R\in\mathcal{R}_{\rm ad}(\nu)}\Gamma(R)\in\R$. Then, the optimal control rule set is given by 
		\begin{align}\label{eq:Roptnu}
			\mathcal{R}_{\rm opt}(\nu):={\{R\in\mathcal{R}_{\rm ad}(\nu):~\Gamma(R)=v^*\}}.
		\end{align}
		
		\begin{remark}\label{rem:valuefcn2s}
			For any strict admissible control $(\alpha,\zeta)\in\U_{\rm ad}[0,T]$, let $X^{\alpha,\zeta}$ be the state process obeyed by \eqref{eq:controlledSDE} and $R$ be the joint law of $(X^{\alpha,\zeta}_{\cdot},\delta_{\alpha_{\cdot}}(\d u)\d t,\zeta_{\cdot})$. Then, $R$ can be easily verified to be an admissible control and $J(\alpha,\zeta)=\Gamma(R)$ with $J(\alpha,\zeta)$ defined in \eqref{cost_func0}. This yields that $v^*\leq\inf_{(\alpha,\zeta)\in\U_{\rm ad}[0,T]}J(\alpha,\zeta)$.
		\end{remark}
		To arrive at the reverse inequality in Remark \ref{rem:valuefcn2s}, for $(t,x,\nu)\in[0,T]\times\R^n\times\Pc_2(\R^n)$, define a subset of $\R^n\times\R^{n\times n}\times\R\times\R^m$ by
		\begin{align}\label{eq:Ktxnu}
			& K(t,x,\nu)\\
			&\quad:=\left\{(b(t,x,\nu,u),\sigma\sigma^{\T}(t,x,\nu,u),z,w);~z\geq f(t,x,\nu,u), w_i\leq \phi^i(t,x,\nu,u),~i\in I,u\in U\right\}. \nonumber
		\end{align}
		We make the following assumption on the set $K(t,x,\nu)$:
		\begin{ass}\label{ass2}
			For any $(t,x,\nu)\in[0,T]\times\R^n\times\Pc_2(\R^n)$, the set $K(t,x,\nu)$ is convex. 
		\end{ass}
Then, we have the following lemma whose proof is provided in \hyperref[appn]{Appendix}.	
\begin{lemma}\label{strictcontrol}
Let	 Assumption~\ref{ass2} hold. For any admissible (relaxed) control rule $R\in\mathcal{R}_{\rm ad}(\nu)$, there exists an admissible control $(\alpha,\zeta)\in\U_{\rm ad}[0,T]$ such that $J(\alpha,\zeta)\leq \Gamma(R)$.
\end{lemma}
		
We next study the existence of optimal relaxed controls. Consider the McKean-Vlasov SDE \eqref{eq:controlledSDE} and define $Y_t:=X_t^{\alpha,\zeta}-\int_0^tG(s)\d\zeta_s$ for $t\in[0,T]$. It follows from the dynamics \eqref{eq:controlledSDE} that
\begin{align*}
\d Y_t=b(t,X_t^{\alpha,\zeta},\Law(X_t^{\alpha,\zeta}),\alpha_t)\d t+\sigma(t,X_t^{\alpha,\zeta},\Law(X_t^{\alpha,\zeta}),\alpha_t)\d W_t,~~Y_0=\xi.  
\end{align*}
Motivated by this expression, we have the next result whose proof is omitted because the necessary part is obvious and the sufficiency can be shown by the classical martingale argument { (c.f. Proposition 3.1 in \cite{Haussmann})}.
\begin{lemma}\label{equivalenceY}
Let $R\in\Pc_2(\Omega;\widehat{\Omega})$. Then, $R\in\mathcal{R}_{ad}(\nu)$ iff there exists an $\Fb$-adapted process $Y=(Y_t)_{t\in [0,T]}$ such that
\begin{itemize}
\item[\rm (i)] $Y$ has continuous trajectories and $R(X_{\cdot}=Y_{\cdot}+\int_0^{\cdot}G(s)\d\zeta_s)=1$; {\rm(ii)} $R\circ Y_0^{-1}=\nu$; 
\item[\rm (iii)] $\E^R\left[\int_0^T\int_U|u|^pq_t(\d u)\d t+|\zeta_T|^p\right]<\infty$; {\rm(iv)} for any $\psi\in C_b^2(\R^n)$, $\tilde{M}_t^R\psi$ is a $(R,\Fb)$-martingale, where $\tilde{M}_t^R\psi:=\psi(Y_t)-\int_0^t\int_U\tilde{\mathbb{G}}\psi(s,X_s,R^X(s),u)(Y_s)q_s(\d u)\d s$ with 
\begin{align*}
\tilde{\mathbb{G}}\psi(t,x,\mu,u)(y):=b(t,x,\mu,u)\nabla_x\psi(y)+\frac12\tr(\sigma\sigma^{\T}(t,x,\mu,u)\nabla_{xx}^2\psi(y)).    
\end{align*}
\item[\rm (v)] for $i\in I$, we have $(m\times R)(A_i^R)=T$.
\end{itemize}
\end{lemma}
		
Moreover, we have
\begin{lemma}\label{lsc}
The objective functional $R\mapsto\Gamma(R)$ defined by \eqref{eq:objectGammaR} is l.s.c. on $\mathcal{R}_{\rm ad}(\nu)$.
\end{lemma}
		
\begin{proof}
Let us define the mapping $\Lambda(R,\omega):\Pc_2(\Omega;\widehat{\Omega})\times\Omega\mapsto\R$ by
\begin{align*}
\Lambda(R,\omega)=\int_0^T\int_Uf(t,X_t(\omega),R^X(t),u)q_t(\omega,\d u)\d t+{\int_{[0,T]} c(t)d\zeta_t(\omega)+g(X_T(\omega),R^X(T))}.  
\end{align*}
{For any $x_n\to x$ in $D^n[0,T]$ under $d_D$, we conclude by $\lambda(T)=0$ for any bijective increasing continuous function $\lambda:[0,T]\to[0,T]$ and the property of the Skorokhod metric $d_D$ that $x_n(T)\to x(T)$.}
Then, for $R\in\Pc_2(\Omega;\widehat{\Omega})$, we get $\Lambda(R,\cdot)$ is l.s.c. by Lemma 3.5 in \cite{Haussmann}. Furthermore, $\Lambda(R,\omega)$ is uniformly l.s.c. in $R\in\Pc_2(\Omega;\widehat{\Omega})$ w.r.t. $\omega\in\Omega$ as $f$ is uniformly l.s.c. in $\mu\in\Pc_2(\R^n)$ with respect to $(t,x,u)\in[0,T]\times\R^n\times\R^l$. Thus, if $R_{\ell}\to R$ in $\Pc_2(\Omega;\widehat{\Omega})$ as $\ell\to\infty$ with $R_{\ell},R\in\mathcal{R}_{\rm ad}(\nu)$, then $R_{\ell}\to R$ in $\Pc_2(\Omega;\widehat{\Omega})$ as $\ell\to\infty$, and hence $R_{\ell}\to R$ as $\ell\to\infty$ with $R_{\ell},R\in\mathcal{R}_{\rm ad}(\nu)$. It then follows that
\begin{align*}
&\liminf_{\ell\to\infty}\Gamma(R_{\ell})=\liminf_{\ell\to\infty}\E^{R_{\ell}}[\Lambda(R_{\ell},\omega)]\\
&\qquad\geq\liminf_{\ell\to\infty}\left\{\E^{R_{\ell}}[\Lambda(R_{\ell},\omega)]-\E^{R_{\ell}}[\Lambda(R,\omega)]\right\}+\liminf_{\ell\to\infty}\E^{R_{\ell}}[\Lambda(R,\omega)]\geq \E^{R}[\Lambda(R,\omega)].
\end{align*}
This yields the desired result that the objective functional $\Gamma$ is l.s.c. on $\mathcal{R}_{\rm ad}(\nu)$.
\end{proof}

The next result also holds, whose proof is reported in \hyperref[appn]{Appendix}.	
\begin{lemma}\label{closedset}
The set $\{R\in\Pc_2(\Omega;\widehat{\Omega});~(m\times R)(A_i^R)=T,~\forall i\in I\}$ is closed in $\Pc_2(\Omega;\widehat{\Omega})$.
\end{lemma}
		
We next show that the $\epsilon$-optimal set is compact.
\begin{lemma}\label{compactness}
For any $\epsilon>0$, let us define the $\epsilon$-optimal set $\mathcal{R}_{\rm opt}^{\epsilon}(\nu)=\{R\in\mathcal{R}_{\rm ad}(\nu);~\Gamma(R)\leq v^*+\epsilon\}$. Then, the set $\mathcal{R}_{\rm opt}^{\epsilon}(\nu)$ is compact in $\Pc_2(\Omega;\widehat{\Omega})$.
\end{lemma}
		
\begin{proof}
We first show that $\mathcal{R}_{\rm opt}^{\epsilon}(\nu)$ is relatively compact in $\Pc_2(\Omega;\widehat{\Omega})$. Let $(R_{\ell})_{\ell\geq1}\subset\mathcal{R}_{\rm opt}^{\epsilon}$ be a sequence of admissible control rules. As in Theorem 3.7 of \cite{Haussmann}, we introduce the auxiliary space $\tilde\Omega=\Omega\times C([0,T];\R^n)$,  $\tilde\F=\F\otimes\F^{\mathcal{C}}$ and $\tilde\F_t=\F_t\otimes\F^{\mathcal{C}}_t$ for $t\in[0,T]$. Here,  $\F^{\mathcal{C}}$ is the Borel $\sigma$-algebra on $C([0,T];\R^n)$ and $\F^{\mathcal{C}}_t=\sigma(y(s);s\leq t,y\in C([0,T];\R^n))$. Denote by $\tilde\omega=({\psi},q,\zeta,y)$ the generic element of $\tilde\Omega$. Define the extended probability measure $\tilde{R}_{\ell}$ on $(\tilde\Omega,\tilde\F)$ by
\begin{align}\label{extension}
\tilde R_{\ell}(F\times C)=\int_F\mathbf{1}_C(y(\omega))R_{\ell}(\d\omega),~~\forall F\in\mathcal{F},~C\in\F^{\mathcal{C}}
\end{align}
with $y(\omega)$ being defined by  $y_{\cdot}(\omega):=\psi_{\cdot}(\omega)-\int_0^{\cdot}G(s)\d\zeta_s(\omega)$.
			
By Lemma~\ref{equivalenceY}-(i), we have $y_{\cdot}$ has continuous trajectories under probability $R_{\ell}$, and hence $\tilde{R}_{\ell}$ is well-defined.	 On the other hand, for $R\in\mathcal{R}_{\rm opt}^{\epsilon}(\nu)$, Assumption~\ref{ass1}-(A5) and (A6) yield that
\begin{align*}
v^*+\epsilon&\geq \E^R\left[\int_0^T\int_Uf(t,X_t,R^X(t),u)q_t(\d u)\d t+{\int_{[0,T]} c(t)d\zeta_t+g(X_T,R^X(T))}\right]\\
&\geq \E^R\Big[-\int_0^TM(1+|X_t|^p+M_2(R^X(t)))\d t+M_1\int_0^T\int_U|u|^pq_t(\d u)\d t+c_0|\zeta_T|\\
&\qquad\quad-M(1+|X_T|^p+M_2(R^X(T)))\Big].
\end{align*}
Consequently, we have
\begin{align}\label{uniformbounded}
\sup_{R\in\mathcal{R}_{\rm opt}^{\epsilon}(\nu)}\E^R\left[\int_0^T\int_U|u|^pq_t(\d u)\d t+\left|\zeta_T\right| \right]<+\infty.
\end{align}
Then, it follows from Corollary 2.9 in \cite{Haussmann} that $(\tilde R_{\ell}|_{\mathcal{A}^k[0,T]})_{\ell\geq1}$ is tight. Moreover, by Proposition B.3 in {\cite{Lacker}} and Kolmogorov's tightness criterion, we can deduce that $(\tilde R_{\ell}|_{\mathcal{Q}\times C([0,T];\R^n)})_{\ell\geq1}$ is relatively compact in $\Pc_2(\mathcal{Q}\times C([0,T];\R^n))$. 
			
We next prove that $(\tilde R_{\ell})_{\ell\geq1}$ is tight. Since both $(\tilde R_{\ell}|_{\mathcal{Q}\times C([0,T];\R^n)})_{\ell\geq1}$ and $(\tilde R_{\ell}|_{\mathcal{A}^k[0,T]})_{\ell\geq1}$ are tight, for any $\epsilon>0$,  there exist some compact subsets $K\subset\mathcal{Q}$, $V\subset\mathcal{A}^k[0,T]$ and $C\subset C([0,T];\R^n)$ such that $\sup_{\ell\geq1}\tilde R_{\ell}(D^n[0,T]\times K\times V\times C)\geq 1-\epsilon$. Let $G(V):=\{\int_0^{\cdot}G(s)\d\zeta;~\zeta\in V\}$ and $D=C+G(V)=\left\{y+g;~y\in C,g\in G(V)\right\}$. Note that $C$ and $G(V)$ are compact in $C([0,T];\R^n)$ and $\mathcal{A}^n[0,T]$, respectively. Then, they are both compact in $D^n[0,T]$ and so is $D$. Moreover, $(\tilde R_{\ell})_{\ell\geq1}$ is supported on $\{x_{\cdot}=y_{\cdot}+\int_0^{\cdot}G(s)\d\zeta\}$, and thus $\sup_{\ell\geq1}\tilde R_{\ell}(D\times K\times V\times C)=\sup_{\ell\geq1}\tilde R_{\ell}(D^n[0,T]\times K\times V\times C)\geq 1-\epsilon$. Here, $D\times K\times V\times C$ is compact in $\Omega$, and hence $(R_{\ell})_{\ell\geq1}$ is tight. Moreover, it follows from \eqref{uniformbounded} that 
\begin{align*}
\sup_{R\in\mathcal{R}_{\rm opt}^{\epsilon}(\nu)}\E^R\left[\sup_{t\in[0,T]}\left|X_t\right|^p+\int_0^T\int_U|u|^pq_t(\d u)\d t+\sup_{t\in[0,T]}\left|Y_t\right|^p\right]<+\infty. 
\end{align*}
Thanks to Corollary B.2 in \cite{Lacker}, we have that $(\tilde{R}_{\ell}|_{\widehat{\Omega}\times C([0,T];\R^n)})_{\ell\geq1}$ is relatively compact in $\Pc_2(\widehat{\Omega}\times C([0,T];\R^n))$. As $(R_{\ell})_{\ell\geq1}$ is the $\Omega$-marginal { of $(\tilde R_{\ell})_{\ell\geq1}$, $\ie$, $R_{\ell}=\tilde R_{\ell}|_{\Omega}$} in the terminology of \cite{Jacod}, $(R_{\ell})_{\ell\geq1}$ is thus relatively compact in $\Pc_2(\Omega;\widehat{\Omega})$, and hence so is $\mathcal{R}_{\rm opt}^{\epsilon}(\nu)$. 
			
Lastly, we show that $\mathcal{R}_{\rm opt}^{\epsilon}(\nu)$ is closed. Let $(R_{\ell})_{\ell\geq1}\subset\mathcal{R}_{\rm opt}^{\epsilon}(\nu)$ satisfy $R_{\ell}\to R$ in $\Pc_2(\Omega;\widehat{\Omega})$ as $\ell\to\infty$. Then, we prove that the limit point $R\in\mathcal{R}_{\rm ad}(\nu)$. In light of Lemma~\ref{equivalenceY} and Lemma~\ref{closedset}, it suffices to verify claims (i) and (iv) of Lemma~\ref{equivalenceY} for the limit $R$. To do it, let $\tilde R_{\ell}$ and $\tilde R$ be the extension measure on $(\tilde\Omega,\tilde\F)$ as in \eqref{extension}. Then, by construction, we have $\tilde{R}_{\ell}\xrightarrow{w}\tilde{R}$ as $\ell\to\infty$.  Since the set $\{\tilde\omega\in\tilde{\Omega};~y_{\cdot}=x_{\cdot}+\int_0^{\cdot}G(s)\d\zeta_s\}$ is closed in $\tilde\Omega$, it holds that 
\begin{align*}
\tilde R\left(\left\{\tilde\omega\in\tilde{\Omega};~y_{\cdot}=x_{\cdot}+\int_0^{\cdot}G(s)\d\zeta_s\right\}\right)\geq \limsup_{\ell\to\infty}\tilde R_{\ell}\left(\left\{\tilde\omega\in\tilde{\Omega};~y_{\cdot}=x_{\cdot}+\int_0^{\cdot}G(s)\d\zeta_s\right\}\right)=1,   
\end{align*}
which proves claim (i) of Lemma~\ref{equivalenceY}. For the martingale argument, because $(b,\sigma)$ are uniformly continuous in $\mu\in\Pc_p(\R^n)$ (actually there exists a uniform Lipschitz constant) w.r.t. $(t,x,u)\in[0,T]\times\R^n\times\R^l$, for any $\psi\in C_b^2(\R^n)$, it holds that 
\begin{equation}\label{uniformcon}
\tilde M_t^{R_{\ell}}\psi(\omega)\to\tilde M_t^R\psi(\omega),\quad \ell\to\infty
\end{equation}
uniformly w.r.t. $(t,\omega)\in[0,T]\times\Omega$. Moreover,  for any bounded $\F_s$-measurable random variable $\eta$ with $0\leq s<t\leq T$, we have
\begin{align*}
&\left|\E^R\left[(\tilde M_t^R\psi(\omega)-\tilde M_s^R\psi(\omega))\eta(\omega)\right]\right|\leq \E^{R_{\ell}}\left[\left|\left(\tilde M_t^R\psi(\omega)-\tilde M_t^{R_n}\psi(\omega)\right)\eta(\omega)\right|\right]\nonumber\\
&\quad+\E^{R_{\ell}}\left[\left|\left(\tilde M_s^R\psi(\omega)-\tilde M_s^{R_{\ell}}\psi(\omega)\right)\eta(\omega)\right|\right]+ \left|\E^{R} \left[\tilde M_t^R\psi(\omega)\eta(\omega)\right]-\E^{R_{\ell}}\left[\tilde M_t^{R}\psi(\omega)\eta(\omega)\right]\right|\\
&\quad+\left|\E^{R}\left[\tilde M_s^R\psi(\omega)\eta(\omega)\right]-\E^{R_{\ell}}\left[\tilde M_s^{R}\psi(\omega)\eta(\omega)\right]\right|+\left|\E^{R_{\ell}}\left[\left(\tilde M_t^{R_{\ell}}\psi(\omega)-\tilde M_s^{R_{\ell}}\psi(\omega)\right)\eta(\omega)\right]\right|\\
&\quad=: I_1+I_2+I_3+I_4+I_5,
\end{align*}
where the term $I_5$ equals to $0$ as $R_{\ell}\in\mathcal{R}_{\rm ad}(\nu)$; while terms $I_1$ and $I_2$ converge to $0$ as $\ell\to\infty$ due to the uniform convergence \eqref{uniformcon}. Lastly, terms $I_3$ and $I_4$ converge to $0$ because $\tilde M_t^R\psi(\omega)\eta(\omega)$ is continuous with at most quadratic growth in $(x,q)$ and $R_{\ell}\to R$ in $\Pc_2(\Omega;\widehat{\Omega})$ as $\ell\to\infty$. Hence, we can conclude that $\E^R[(\tilde M_t^R\psi(\omega)-\tilde M_s^R\psi(\omega))h(\omega)]=0$. This implies claim (iv) of Lemma~\ref{equivalenceY}. 
\end{proof}
		
We are ready to present the main result of this section:
\begin{theorem}\label{existence_opt}
Let Assumption~\ref{ass1} hold. Then, the set of optimal relaxed control rules $\mathcal{R}_{\rm opt}(\nu)$ defined by \eqref{eq:Roptnu} is a (nonempty) compact and convex subset of $\Pc_2(\Omega;\widehat{\Omega})$. Consequently, there exists an optimal relaxed control to problem \eqref{eq:objectGammaR}.
\end{theorem}
		
\begin{proof}
Note that $v^*=\inf_{R\in\mathcal{R}_{\rm ad}(\nu)}\Gamma(R)=\inf_{R\in\mathcal{R}_{\rm opt}^{\epsilon}(\nu)}\Gamma(R)$. Then, it follows from Lemma~\ref{lsc} and Lemma~\ref{compactness} that $\mathcal{R}_{\rm opt}$ is nonempty. It also can be easily checked to be convex and closed. As a closed subset of the compact set $\mathcal{R}_{\rm opt}^{\epsilon}(\nu)$, it is hence compact.
\end{proof}
		
We also have the following corollary as a direct consequence of Lemma~\ref{strictcontrol} and Theorem~\ref{existence_opt}: 
\begin{corollary}\label{optimality_exist}
Under Assumption~\ref{ass1} and Assumption~\ref{ass2}, there exists an optimal strict control $(\alpha,\zeta)$ to the constrained MFC problem \eqref{cost_func0}. 
\end{corollary}

\section{SMP by the Lagrange Multipliers Method}\label{sec:SMP}
		
The aim of this section is to characterize the optimal control by establishing the stochastic maximum principle (SMP) of the constrained MFC problem \eqref{eq:controlledSDE}-\eqref{cost_func0} with singular control. 
		
Let us first introduce some basic spaces of functions and processes.
Let $\X$ be the space $L^2([0,T];\R^l)/\mathord{\sim}$ equipped with the inner product $\langle\gamma^1,\gamma^2\rangle_{\X}=\int_0^T\langle\gamma_t^1,\gamma_t^2\rangle\d t$. Then, $(\X,\langle\cdot,\cdot\rangle_{\X})$ is a Hilbert space.  Let $\Lb_{\Fb}^{2,{l}}$ be the set of $\Fb$-adapted processes $\beta=(\beta_t)_{t\in[0,T]}$ taking values in $\R^l$ such that $\E[\int_0^T|\beta_t|^2\d t]<\infty$. As a result, $\Lb_{\Fb}^{2,{l}}$ is a Hilbert space equipped with the inner product $\langle \beta^1,\beta^2\rangle_{\Lb_{\Fb}^2}=\E[\int_0^T\langle\beta^1_t,\beta^2_t\rangle\d t]$ for $  \beta^1,\beta^2\in\mathbb{L}_{\Fb}^{2,{l}}$. It can be easily verified that $\U[0,T]$ is closed and convex in $(\Lb_{\Fb}^{2,{l}},\langle\cdot,\cdot\rangle_{\Lb_{\Fb}^{2,{l}}})$. Define $\mathbb{M}^2$ ($\subset\Lb_{\Fb}^{2,{n}}$) as the set of $\Fb$-adapted (continuous) $\R^n$-valued process $\eta=(\eta_t)_{t\in[0,T]}$ satisfying the It\^o's representation that 
		\begin{align}\label{eq:ito-rep}
			\eta_t=\eta_0+\int_0^t\eta^1_s\d s+\int_0^t\eta_s^2\d W_s, 
		\end{align}
		where $\eta_0\in L^2(\Omega;\R^n)$ is $\F_0$-measurable,  $(\eta^1,\eta^2)=(\eta_t^1,\eta_t^2)_{t\in [0,T]}\in \Lb_{\Fb}^{2,n}\times\Lb_{\Fb}^{n\times r}$ take values respectively in $\R^n$ and $\R^{n\times r}$. Thus, $\mathbb{M}^2$ is a Hilbert space under the inner product $\langle\cdot,\cdot\rangle_{\mathbb{M}^2}$ defined by
		\begin{align}\label{eq:innerM2}
			\langle \eta,\delta\rangle_{\mathbb{M}^2}:=\E\left[\langle\eta_0,\delta_0\rangle+\int_0^T\left(\langle\eta_t^1,\delta_t^1\rangle+\tr\left(\eta_t^2(\delta_t^2)^{\T}\right)\right)\d t \right],\quad\forall\eta,\delta\in\mathbb{M}^2.    
		\end{align}	
We use $\Lb_{\Fb,\R}^2$ to represent the set of $\Fb$-adapted processes in $L^2([0,T]\times\Omega;\R)/\mathord{\sim}$. Then, $\Lb^2_{\Fb,\R}$ is a Hilbert space with the inner product $\langle\beta^1,\beta^2\rangle_{\Lb^2_{\Fb,\R}}:=\E[\int_0^T\beta_t^1\beta_t^2\d t]$ for $\beta^1,\beta^2\in\Lb^2_{\Fb,\R}$. 
		
To adopt the methodology introduced in \cite{BWY}, we propose to embed the state process into a suitable Banach space. Without loss of generality, we can assume that for any singular control $\zeta\in\mathcal{A}^k[0,T]$, the associated {Radon} measures $\mu_i$ for $i\in I$ of the trajectory $t\mapsto\zeta_t(\omega)$ is singular w.r.t. {the} Lebesgue measure. Otherwise, we can decompose $\d\zeta_t$ using the Lebesgue decomposition w.r.t. the Lebesgue measure and absorb the absolutely continuous (regular) term to the drift $b$. {Indeed, suppose that $d\zeta_t$ admits the decomposition $d\zeta_t = \zeta_t^{\rm a} dt + d\zeta_t^{\rm s}$, where $d\zeta_t^{\rm s}$ is a non-decreasing process that is singular with respect to the Lebesgue measure on $[0,T]$. Then, by setting $\tilde\alpha = (\alpha, \zeta^{\rm a})$ as the new regular control and $\tilde\zeta = \zeta^{\rm s}$ as the new singular control, and by defining the new drift coefficient as $\tilde b(t,x,\mu,u,u') = b(t,x,\mu,u) + G(t)u'$, we can consider the problem where the singular control is singular with respect to the Lebesgue measure.}

Consider the following space:
\begin{align}\label{eq:mathscrK}
\mathscr{X}&:=\Big\{X=(X_t)_{t\in[0,T]};~X \text{ is }\Fb\text{-adapted process with $X_0\in L^2(\Omega;\R^n)$ and }\nonumber\\
&\qquad~ X_{\cdot}=X_0+\int_0^{\cdot} X_s^1\d s+\int_0^{\cdot} X_s^2\d W_s+\int_0^{\cdot}G(s)\d X_s^3\nonumber\\
&\qquad~ \text{with}~(X^1,X^2,X^3)\in \Lb_{\Fb}^{2,{n}}\times\Lb_{\Fb}^{2,{n\times r}}\times\mathcal{A}^k[0,T] \Big\}.
\end{align}
Consequently, $\mathscr{X}$ is a Banach space under the norm $\|\cdot\|_{\mathscr{X}}$ given by
\begin{align*}
\|X\|_{\mathscr{X}}:=\sqrt{\|X^1\|_{\Lb_{\Fb}^{2,{n}}}^2+\|X^2\|_{\Lb_{\Fb}^{2,{n\times r}}}^2}+\sqrt{\E\left[\left|\int_0^TG(s)\d X_s^3\right|^2\right]},~~\forall X\in\mathscr{X}.  
\end{align*}
		
By the assumption that, for any $X\in\mathscr{X}$, its 3rd component $X^3$ of the decomposition in $\mathscr{X}$ is singular w.r.t. to {the} Lebesgue measure. We further define the set $\mathbb{G}$ and the norm $\|\cdot\|_{L^2(\Omega;\Mc([0,T])^n)}$ by
\begin{align*}
&\Gb:=\Big\{S=(S_t)_{t\in[0,T]};~ S_{\cdot}=\int_0^{\cdot}G(s)\d\eta_s~\text{for}~\eta=(\eta_t)_{t\in[0,T]}\in\mathcal{A}^k[0,T]\\
&\qquad\qquad\qquad\text{with $\d\eta_t^i$ being singular {\rm w.r.t.} the Lebesgue measure on $[0,T]$}\Big\},\\
\|S\|_{L^2(\Omega;\Mc([0,T])^n)}&:=\left\{\E\left[\left(\sum_{i=1}^n\|S^i\|_{[0,T];{\rm FM}}\right)^2\right]\right\}^{\frac12}\nonumber\\
&=\left\{\E\left[\left|\sup_{\substack{\ell:[0,T]\to\R^n\\\|\ell^i\|_{\infty}+\Lip(\ell^i)\leq 1 }}\int_{[0,T]}\ell(t)^{\T}\d S_t\right|^2\right]\right\}^{\frac12}.\\
\end{align*}
Here, the last equality follows by recalling the definition of the Foutet-Mourier norm. { In view of (B6) of Assumption~\ref{ass3}, we derive that $G(t)^{\top}G(t)$ is invertible for every $t\in [0,T]$. Therefore, for any mapping $h:[0,T]\to\R^k$ such that every component $h^i:[0,T]\times\R$ satisfies $\|h^i\|_{\infty}+\Lip(h^i)\leq 1$, there exists a mapping $\ell:[0,T]\to\R^n$ such that $G(t)^{\T}\ell(t)=h(t)$ holds for all $t\in [0,T]$. Indeed, we can choose $\ell(t)=G(t)(G(t)^{\T}G(t))^{-1}h(t)$. By using the fact that $A^{-1}-B^{-1}=A^{-1}(B-A)B^{-1}$ and $|y^{\top}(G(t)^{\T}G(t))^{-1}y|\leq c^{-1}|y|^2$ for $y\in\R^k$, we deduce that $|\ell (t)-\ell(t)|\leq C|t-s|$ for some constant $C>0$ only depending on $c$. Consequently, it holds that
\begin{align*}
\|\eta\|_{L^2(\Omega;\Mc([0,T])^k)}^2&=\E\left[\left|\sup_{\substack{h:[0,T]\to\R^k\\\|h^i\|_{\infty}+\Lip(h^i)\leq 1 }}\int_{[0,T]}h(t)^{\T}\d \eta_t\right|^2\right]\\
    &=\E\left[\left|\sup_{\substack{\ell:[0,T]\to\R^n\\\|\ell^i\|_{\infty}+\Lip(\ell^i)\leq 1 }}\int_{[0,T]}\ell(t)^{\T}G(t)\d \eta_t\right|^2\right]\leq C\|S\|_{L^2(\Omega;\Mc([0,T])^n)},
\end{align*}
where, $S(t)=\int_0^tG(t)\d\eta_t$ for $t\in [0,T]$ and the constant $C>0$ only depends on $(c,\|G\|_{\infty})$. This implies that $\Gb$ is a closed convex cone of the Banach space $L^2(\Omega;\mathcal{M}([0,T])^n)$.
} Thus, $\mathscr{X}$ can be viewed as a closed convex cone of the Banach space $\mathbb{T}^2:=\M^2\times L^2(\Omega;({\cal M}[0,T])^n)$ under the product norm. For any $\lambda=(\lambda^1,\lambda^2)\in\mathbb{T}^{2,*}\simeq \M^2\times L^2(\Omega;(C([0,T];\R)^{**})^n)$, the dual pair of $X=L+S\in\mathscr{X}$ for $(L,S)\in\M^2\times\Gb^2$ is defined by
		\begin{align*}
			\langle X,\lambda\rangle_{\mathbb{T}^2,\mathbb{T}^{2,*}}=\langle L,\lambda^1\rangle_{\M^2}+\langle S,\lambda^2\rangle_{L^2(\Omega;({\cal M}[0,T])^k),L^2(\Omega;(C([0,T];\R)^{**})^k)}.   
		\end{align*}
Following Lemma 3.6 in \cite{BWY}, let us introduce	
\begin{align}\label{eq:XYHexam}
\tX:=\mathbb{T}^2\times\Lb_{\Fb}^2\times L^2(\Omega;({\cal M}[0,T])^n),\quad \tY_{i}:=\Lb^2_{\Fb,\R},~ \forall i\in I,\quad\tH:=\mathbb{T}^2.
\end{align}
The norm for $\tX$ is defined to be the (default) product norm and we take $\|\cdot\|_{\tY_i}$~$(i\in I)$ as  the norm $\|\cdot\|_{\Lb^2_{\Fb,\R}}$. Moreover, define $J:\tX\to\R$, $\Phi_i:\tX\to\tY_{i}$~$(i\in I)$ and $F:\tX\to\tH$ by, for $(X,\alpha,\zeta)\in\tX$,
\begin{align}\label{func_def}
J(X,\alpha,\zeta) &:=	\E\left[\int_0^Tf(t,X_t,\Law(X_t),\alpha_t)\d t+{\int_{[0,T]} c(t)d\zeta_t+g(X_T,\mu_T)} \right],\nonumber\\
\Phi_i(X,\alpha,\zeta)(t)&:=-\phi^i(t,X_t,\Law(X_t),\alpha_{t}),~~~\forall t\in[0,T],~i\in I,\\
F(X,\alpha,\zeta)(\cdot)&:=X_{\cdot}-\xi-\int_0^{\cdot}b(s,X_s,\Law(X_s),\alpha_s)\d s-\int_0^{\cdot}\sigma(s,X_s,\Law(X_s),\alpha_s)\d W_s-\int_0^{\cdot}G(s)\d\zeta_s.\nonumber
\end{align}
The well-posedness of $\Phi_i$ is guaranteed by Assumption~\ref{ass1}-{\rm (A2)}. Thus, we can formulate the following optimization problem (equivalent to the constrained MFC problem \eqref{cost_func0}) with constraints that
\begin{align}\label{optimization}
\begin{cases}
\displaystyle       \inf~J(X,\alpha,\zeta)\\[0.4em]
\displaystyle       \st~~(X,\alpha,\zeta)\in C:=\mathscr{X}\times\U[0,T]\times\mathcal{A}^k[0,T],\\[0.4em]
\displaystyle      ~~~~~~~  \Phi_i(X,\alpha,\zeta)\leq_{K}0,~~\forall i\in I,\\[0.4em]
\displaystyle  ~~~~~~~F(X,\alpha,\zeta)=0,
\end{cases}
\end{align}
		where $C$ in \eqref{optimization} is closed and convex in $\tX$, and $K$ is a closed and convex cone given by 
		\begin{align}\label{eq:K12}
			\displaystyle  K:=\{f\in \Lb_{\Fb,\R}^2;~f(t,\omega)\geq 0,~\forall (t,\omega)\in[0,T]\times\Omega\}=:\Lb_{\Fb,\R+}^2.
		\end{align}

   {By Theorem~\ref{existence_opt} and Corollary~\ref{optimality_exist}, the existence of an optimal control is guaranteed. Moreover, as discussed in \cite{BWY}, the Fritz–John (FJ) condition provides a necessary condition for the optimal regular controls. In the sequel, we shall extend this methodology by introducing the FJ condition for C-MFC problem with singular controls.} By using Lemma 3.11 in \cite{BWY}, the FJ optimality condition holds for the constrained optimization problem \eqref{optimization} in the following sense.
\begin{lemma}[FJ optimality condition for problem \eqref{optimization}]\label{lem:FJCondAbstractprob}
Sppose that Assumption~ref{ass3} holds.
Let $(\widehat{X},\widehat{\alpha},\widehat{\zeta})\in C$ be an optimal pair to problem \eqref{optimization}. Then,  there exist multipliers $r_0\geq0$, $\eta^i\in\Lb^{2}_{\Fb,\R+}$ for $i\in I$ and $\lambda=(\lambda_1,\lambda_2)\in\M^2\times L^2(\Omega;(C([0,T],\R)^{**})^n)$, which are not all zeros, such that
\begin{align*}
&\E\left[\int_0^T\Phi_i(\widehat{X},\widehat{\alpha},\widehat{\zeta})(t)\eta^i_t\d t\right]=0,~~\forall i\in I,\\
& r_0+\sum_{i\in I}\|\eta^i\|_{\Lb^{2}_{\Fb,\R}}+\|\lambda\|_{\M^2\times L^2(\Omega;(C([0,T],\R)^{**})^n)}=1,\\
& -\xi_1= r_0 D_XJ(\widehat{X},\widehat{\alpha},\widehat{\zeta})(\cdot)+\sum_{i\in I}\E\left[\int_0^TD_X\Phi_i(\widehat{X},\widehat{\alpha},\widehat{\zeta})(\cdot)(t)\eta^i_t\d t\right]\nonumber\\
&\qquad\quad-\left\langle \lambda, D_XF(\widehat{X},\widehat{\alpha},\widehat{\zeta})(\cdot)\right\rangle_{\tH^*,\tH}~~\text{\rm in}~\mathbb{T}^2,\\
& -\xi_2= r_0 D_{\alpha}J(\widehat{X},\widehat{\alpha},\widehat{\zeta})(\cdot)+\sum_{i\in I}\E\left[\int_0^TD_{\alpha}\Phi_i(\widehat{X},\widehat{\alpha},\widehat{\zeta})(\cdot)(t)\eta^i_t\d t\right]\nonumber\\
&\qquad\quad-\left\langle \lambda, D_{\alpha}F(\widehat{X},\widehat{\alpha},\widehat{\zeta})(\cdot)\right\rangle_{\tH^*,\tH}~~\text{\rm in}~\Lb_{\Fb}^{2,{l}},\\
& -\xi_3= r_0 D_{\zeta}J(\widehat{X},\widehat{\alpha},\widehat{\zeta})(\cdot)+\sum_{i\in I}\E\left[\int_0^TD_{\zeta}\Phi_i(\widehat{X},\widehat{\alpha},\widehat{\zeta})(\cdot)(t)\eta^i_t\d t\right]\\
& \quad\qquad-\left\langle \lambda, D_{\zeta}F(\widehat{X},\widehat{\alpha},\widehat{\zeta})(\cdot)\right\rangle_{\tH^*,\tH}~~\text{\rm in}~ L^2(\Omega;(C([0,T];\R)^{**})^n),
\end{align*}
{where $\xi=(\xi_1,\xi_2,\xi_3)\in N_C(\widehat{X},\widehat{\alpha},\widehat{\zeta})\subset\tX^*$ with the normal cone $N_C(\widehat{X},\widehat{\alpha},\widehat{\zeta})$ defined by 
\begin{align}\label{eq:cnormalconne}
N_C(\widehat{\boldsymbol{x}}):=\{\xi\in\mathtt{X}^*:\langle\xi,\widehat{\boldsymbol{x}}-\boldsymbol{x}\rangle_{\mathtt{X}^*,\mathtt{X}}\geq 0,~\forall \boldsymbol{x}\in C\},\quad \widehat{\boldsymbol{x}}\in \mathtt{X},
\end{align} 
and $D_{\zeta}$ denotes the Fr\'{e}chet derivative with respect to $\zeta$ in the Banach space $\mathcal{M}([0,T])^k$.}
\end{lemma}
		
Then, we have the following main result in this section. 
\begin{theorem}[SMP for constrained MFC problem \eqref{cost_func0} with singular control]\label{SMP_con} 
Sppose that Assumption~ref{ass3} holds.
Let $(\widehat{X},\widehat{\alpha},\widehat{\zeta})\in C$ be an optimal solution to problem \eqref{cost_func0}. There exist some $r_0\geq 0$ and $\eta^i\in \Lb^2_{\Fb,\R+}$ for $i\in I$ such that 
\begin{itemize}
\item [{\rm(i)}] for $i\in I$, $\eta_t^i\geq 0$, $\d t\times\d \Pb$-$\as$ and $\eta^i\in\Lb^{2}_{\Fb,\R+}$, let $A^j=(A_t^i)_{t\in[0,T]}$ be the increasing process that $A_t^i=\int_0^t\eta_s^i\d s$. Then, for $\Pb$-$\as$ $\omega$, the process $A^i(\omega)$ only increases on the set $\{t\in[0,T];~\phi^j(t,\widehat{X}_t(\omega),\widehat{\mu}_t,\widehat{\alpha}_t(\omega))=0\}$, $\d t$-$\as$, with $\widehat{\mu}_t=\Law(\widehat{X}_t)$; 
\item [{\rm(ii)}] it holds that $\displaystyle r_0+\sum_{i\in I}\|\eta^i\|_{\Lb^{2}_{\Fb,\R}}=1$; 
\item [{\rm(iii)}]  $\Pb\big(r_0c_i(t)+G_i(t)^{\T}Y_t\geq 0,~\forall (t,i)\in[0,T]\times\{1,\ldots,k\}\big)=1$ with $c_i(t)$ and $G_i(t)$ being the $i$-th row of $c(t)$ and $G(t)$, respectively. It also holds that
\begin{align}\label{completeness}
\Pb\left(\sum_{i=1}^k\mathbf{1}_{\left\{r_0c_i(t)+G_i(t)^{\T}Y_t>0\right\}}\d\widehat{\zeta}_t^i=0\right)=1.
\end{align} 
\end{itemize}
Furthermore, the stochastic {\rm(}first-order{\rm)} minimum condition holds that,  $\d t\times\d \Pb$-a.s. 
			\begin{align}\label{SMP}
				&\left\langle \nabla_u H^{r_0}(t,\widehat{X}_t,\widehat{\mu}_t,\widehat{\alpha}_t,Y_t,Z_t)-\sum_{i\in I}\nabla_u\phi^i(t,\widehat{X}_t,\widehat{\mu}_t,\widehat{\alpha}_t)\eta_t^i, u-\widehat{\alpha}_t\right\rangle\geq 0,\quad \forall u\in U.
			\end{align}
			The Hamiltonian $H^{r_0}:[0,T]\times\R^n\times\R^l\times\R^n\times\mathcal{P}_2(\R^n)\times\R^{n\times r}\mapsto \R$ is defined  by
			\begin{align}\label{Hamiltonian}
				H^{r_0}(t,x,u,\mu,y,z):=\langle b(t,x,\mu,u),y\rangle+\tr\left(\sigma(t,x,\mu,u)z^{\T}\right)+r_0f(t,x,\mu,u),
			\end{align}
and the process pair $(Y,Z)=(Y_t,Z_t)_{t\in[0,T]}$ taking values in $\R^n\times\R^{n\times r}$ is the unique solution to the  constrained BSDE:
\begin{align}\label{BSDE}
Y_t&={r_0\nabla_xg(\widehat{X}_T,\widehat{\mu}_T)+r_0\E'\left[\pa_{\mu}g(\widehat{X}_T',\widehat{\mu}_T)(\widehat{X}_T)\right]}\nonumber\\
&\quad+\int_t^T\nabla_x H^{r_0}(s,\widehat{X}_s,\widehat{\mu}_s,\widehat{\alpha}_s,Y_s,Z_s)\d s+\int_t^T\E'\left[\pa_{\mu}H^{r_0}(s,\widehat{X}_s',\widehat{\mu}_s,\widehat{\alpha}_s',Y_s',Z_s')(\widehat{X}_s)\right]\d s\nonumber\\
&\quad-\sum_{i\in I}\int_t^T\nabla_x\phi^i(s,\widehat{X}_s,\widehat{\mu}_s,\widehat{\alpha}_s)\eta^i_s\d s-\sum_{i\in I}\int_t^T\E'\left[\pa_{\mu}\phi^i(s,\widehat{X}_s',\widehat{\mu}_s,\widehat{\alpha}_s')(\widehat{X}_s){\eta^i_s}'\right]\d s\nonumber\\
&\quad-\int_t^TZ_s\d W_s.
\end{align}
Here, $(\widehat{X}',\widehat{\alpha}',Y',Z',\eta^{i'})$ is  an independent copy of $(\widehat{X},\widehat{\alpha},Y,Z,\eta^{i})$ defined on the probability space $(\Omega',\F',P')$.
\end{theorem}
		
\begin{proof}
 By using definition of the normal cone $N_C(\widehat{X},\widehat{\alpha},\widehat{\zeta})$ {(see \eqref{eq:cnormalconne})}, for all $(X,\alpha,\zeta)\in C$ and $\xi=(\xi_1,\xi_2,\xi_3)\in N_C(\widehat{X},\widehat{\alpha},\widehat{\zeta})$,
\begin{align*}
\langle \xi_1,\widehat{X}-X\rangle_{\mathbb{T}^2}+\langle\xi_2,\widehat{\alpha}-\alpha\rangle_{\Lb_{\Fb}^{2,{l}}}+\langle\xi_3,\zeta-\widehat{\zeta}\rangle_{(C([0,T];\R)^{**})^n,({\cal M}[0,T])^n} \geq 0.    
\end{align*}
{Firstly, we claim that $r_0$ and $(\eta^i)_{i\in I}$ cannot be zero simultaneously. Otherwise, we assume that $r_0=0$ and $\eta^i=0$ for $i\in I$ as the 1st step in the proof of Theorem 3.12 in \cite{BWY}. Therefore, it holds by setting $\widehat{\alpha}=\alpha$, $\widehat{\zeta}=\zeta$ and Lemma~\ref{lem:FJCondAbstractprob} that,  for every $X\in\mathscr{X}$,
{\small\begin{align*}
0&\leq \langle\xi_1,\widehat{X}-X\rangle_{\mathtt{T}^2}=\langle\lambda,D_XF(\widehat{X},\widehat{\alpha})(X-\widehat{X})\rangle_{\mathtt{H}^*,\mathtt{H}}\\
&=X_t-\widehat{X}_t-\int_0^t\nabla_xb(s,\widehat{X}_s,\widehat{\mu}_s,\widehat{\alpha}_s)^{\T}(X_s-\widehat{X}_s)\d s-\int_0^t\nabla_x\sigma(s,\widehat{X}_s,\widehat{\mu}_s,\widehat{\alpha}_s)^{\T}(X_s-\widehat{X}_s)\d W_s\\
&\quad-\int_0^t\E'\left[\pa_{\mu}b(s,\widehat{X}_s,\widehat{\mu}_s,\widehat{\alpha}_s)(\widehat{X}_s')^{\T}(X_s'-\widehat{X}_s')\right]\d s-\int_0^t\E'\left[\pa_{\mu}\sigma(s,\widehat{X}_s,\widehat{\mu}_s,\widehat{\alpha}_s)(\widehat{X}_s')^{\T}(X_s'-\widehat{X}_s')\right]\d W_s.
\end{align*}}By setting $L = X - \widehat{X}$ and allowing $L$ to vary arbitrarily in $\mathbb{M}^2$, we conclude that $\langle \lambda, D_X F(\widehat{X}, \widehat{\alpha})(L) \rangle_{\mathtt{H}^*, \mathtt{H}} = 0$ for all  $L\in\mathbb{M}^2$. In light of Lemma 3.9 and Lemma 3.10 of \cite{BWY}, we deduce that $\lambda=0$ that contradicts with the non-degeneracy of the multipliers.  Hence, the claim holds and we can scale the multipliers such that {(ii)} holds.} Then, we can conclude from the structure of Lions derivative (c.f. Section 5.2.1 in \cite{Carmona2015}) and Lemma~\ref{lem:FJCondAbstractprob} that
\begin{align}\label{Xoptimality}
0&\leq{ \E\left[\left\langle \nabla_xg(\widehat{X}_T,\widehat{\mu}_T),L_T\right\rangle+\E'\left[\left\langle \pa_{\mu}g(\widehat{X}_T,\widehat{\mu}_T)(\widehat{X}_T'),L_T'\right\rangle\right]\right]}\nonumber \\
&\quad+r_0\E\left[\int_0^T\langle \nabla_xf(t,\widehat{X}_t,\widehat{\mu}_t,\widehat{\alpha}_t),L_t\rangle\d t+\E'\left[\int_0^T\langle\pa_{\mu}f(t,\widehat{X}_t,\widehat{\mu}_t,\widehat{\alpha}_t)(\widehat{X}_t'),L_t'\rangle\d t\right]\right]\nonumber\\
&\quad-\sum_{i\in I}\E\left[\int_0^T\left\langle\nabla_x\phi^i(t,\widehat{X}_t,\widehat{\mu}_t,\widehat{\alpha}_t)+\E'\left[\pa_{\mu}\phi^i(t,\widehat{X}_t,\widehat{\mu}_t,\widehat{\alpha}_t)(\widehat{X}_t')\right],L_t'\right\rangle\eta^i_t\d t\right]\nonumber\\
&\quad-\E\left[\langle\lambda_1(0),L_0\rangle+\int_0^T\langle\lambda_1^1(t),L_t^1-\nabla_xb(t,\widehat{X}_t,\widehat{\mu}_t,\widehat{\alpha}_t)L_t\rangle\d t\right.\nonumber\\
&\quad+\int_0^T\tr\left(\lambda_1^2(t)(L_t^2-\nabla_x\sigma(t,\widehat{X}_t,\widehat{\mu}_t,\widehat{\alpha}_t)L_t)^{\T}\right)\d t \nonumber\\
&\quad+\int_0^{T}\left\langle\lambda_1^1(t),\E'\left[\pa_{\mu}b(t,\widehat{X}_t,\widehat{\mu}_t,\widehat{\alpha}_t)(\widehat{X}_t')L_t'\right]\right\rangle\d t\nonumber\\
&\quad+\left.\int_0^{T}\tr\left(\lambda_1^2(t)\left(\E'\left[\pa_{\mu}\sigma(t,\widehat{X}_t,\widehat{\mu}_t,\widehat{\alpha}_t)(\widehat{X}_t')L_t'\right]\right)^{\T}\right)\d t\right]\nonumber\\
&\quad+\left\langle \lambda_2,\nu\right\rangle  _{(C([0,T];\R)^{**})^n,(\Mc[0,T])^n},
\end{align}
where $\lambda_1\in\M^2$ has the decomposition $\lambda_1(\cdot)=\lambda_1(0)+\int_0^{\cdot}\lambda_1^1(s)\d s+\int_0^{\cdot}\lambda_1^2(s)\d W_s$, $L=X-\widehat{X}$, with $X\in\mathscr{X}$ has the representation  $L_{\cdot}=L_0+\int_0^{\cdot}L_s^1\d s+\int_0^{\cdot}L_s^2\d W_s+\int_0^{\cdot}G(s)\d L_s^3$, and $\nu$ is the associated radon measure of the mapping $\int_0^{\cdot}G(s)\d L_s^3\in \BV^n[0,T]$.
Define an element $p=(p_t)_{t\in[0,T]}$ in $\M^2$ that, for $t\in[0,T]$,		
\begin{align}\label{adjointp}
p_t&={r_0\nabla_xg(\widehat{X}_T,\widehat{\mu}_T)+r_0\E'\left[\pa_{\mu}g(\widehat{X}_T',\widehat{\mu}_T)(\widehat{X}_T)\right]}\nonumber\\
&\quad-\sum_{i\in I}\int_t^T\left(\nabla_x\phi^i(s,\widehat{X}_s,\widehat{\mu}_s,\widehat{\alpha}_s)+\E'\left[\pa_{\mu}\phi^i(s,\widehat{X}_s',\widehat{\mu}_s,\widehat{\alpha}_s')(\widehat{X}_s){\eta^i_s}'\right]\right)\d s\nonumber\\
&\quad+\int_t^T\left[r_0\nabla_xf(s,\widehat{X}_s,\widehat{\mu}_s,\widehat{\alpha}_s)+\nabla_xb(s,\widehat{X}_s,\widehat{\mu}_s,\widehat{\alpha}_s)\lambda_1^1(s)+\nabla_x\sigma(s,\widehat{X}_s,\widehat{\mu}_s,\widehat{\alpha}_s)\lambda_1^2(s) \right]\d s\nonumber\\
&\quad+\int_t^T\E'\left[r_0\pa_{\mu}f(s,\widehat{X}_s',\widehat{\mu}_s,\widehat{\alpha}_s')(\widehat{X}_s)+\pa_{\mu}b(s,\widehat{X}_s',\widehat{\mu}_s,\widehat{\alpha}_s')(\widehat{X}_s){\lambda_1^1(s)}'\right.\nonumber\\
&\qquad\qquad+\left.\pa_{\mu}\sigma(s,\widehat{X}_s',\widehat{\mu}_s,\widehat{\alpha}_s')(\widehat{X}_s){\lambda_1^2(s)}' \right]\d s-\int_t^T\lambda_1^2(s)\d W_s,
\end{align}
where $MN:=\sum_{i=1}^rM_iL_i\in \R^n$ for any $M=(M_1,\ldots,M_r)\in\R^{n\times n\times r}$ and $L=(L_1,\ldots,L_r)\in\R^{n\times r}$. It then holds that
\begin{align}\label{equationL}
\langle \lambda_1^1(t),L_t^1\rangle\d t&=\langle\lambda_1^1(t)-p(t),\d L_t\rangle+\langle p(t),\d L_t\rangle-\langle\lambda^1_1(t),L^2_t\d W_t\rangle\nonumber\\
&\quad-\langle\lambda_1^1,\nu\rangle_{(C([0,T];\R)^{**})^n,(\Mc[0,T])^n}.
\end{align}
By using \eqref{adjointp} together with It\^{o}'s rule, we can conclude by Fubini's theorem that
\begin{align*}
&\E\left[\int_0^T\langle p_t,\d L_t\rangle\right]=-\E\left[\int_0^T\langle L_t,\d p_t\rangle\right]-\E\left[\langle p,L\rangle_T\right]{+\E\left[\langle p_T,L_T\rangle\right]-\E\left[\langle p_0,L_0\rangle\right]}\\
&=\E\left[\int_0^T\left\langle r_0\nabla_xf(t,\widehat{X}_t,\widehat{\mu}_t,\widehat{\alpha}_t),L_t\right\rangle\d t+\int_0^T\left\langle r_0\E'\left[\pa_{\mu}f(t,\widehat{X}_t',\widehat{\mu}_t,\widehat{\alpha}_t')(\widehat{X}_t)\right],L_t\right\rangle\d t \right]\\
&+\E\left[\int_0^T\left(\langle\lambda_1^1(t),\nabla_xb(t,\widehat{X}_t,\widehat{\mu}_t,\widehat{\alpha}_t)L_t\rangle+\tr\left(\lambda_1^2(t)\left(\nabla_x\sigma(t,\widehat{X}_t,\widehat{\mu}_t,\widehat{\alpha}_t)L_t\right)^{\T}\right)\right)\right]\\
&+\E\left[\int_0^{T}\left\langle\lambda_1^1(t),\E'\left[\pa_{\mu}b(t,\widehat{X}_t',\widehat{\mu}_t,\widehat{\alpha}_t')(\widehat{X}_t)L_t\right]\right\rangle\d t-\langle p_0,L_0\rangle\right]-\E\left[\tr\left(\lambda_1^2(t)L_t^2\right)\d t\right]\\
&+\E\left[\int_0^{T}\tr\left(\lambda_1^2(t)\left(\E'\left[\pa_{\mu}\sigma(t,\widehat{X}_t',\widehat{\mu}_t,\widehat{\alpha}_t')(\widehat{X}_t)L_t\right]\right)^{\T}\right)\d t\right]\\
&-\sum_{i\in I}\E\left[\int_0^T\left\langle \nabla_x\phi^i(t,\widehat{X}_t,\widehat{\mu}_t,\widehat{\alpha}_t),L_t\right\rangle\d t+\int_0^T\left\langle \E'\left[\pa_{\mu}\phi^i(t,\widehat{X}_t',\widehat{\mu}_t,\widehat{\alpha}_t')(\widehat{X}_t)\right],L_t\right\rangle\d t \right].
\end{align*}
Inserting \eqref{equationL} together with the above equality into \eqref{Xoptimality}, we derive
\begin{align*}
\E\left[\langle p_0-\lambda_1(0), L_0\rangle+\int_0^T\langle p_t-\lambda_1^1(t),\d L_t\rangle+\langle\lambda_2-\lambda_1^1,\nu\rangle_{(C([0,T];\R)^{**})^n,(\Mc[0,T])^n} \right]\geq0.
\end{align*}
By setting $\widehat{X}^3=X^3$ and $L_t=-(p_0-\lambda_1(0))-\int_0^t(p_s-\lambda_1^1(s))\d s$, we deduce that $p_0=\lambda_1(0)$ and $p$ is indistinguishable from $\lambda_1^1$ by the arbitrariness of $L$.  As a result, $\lambda_1^1$ solves \eqref{BSDE}. Moreover, by letting $X^1=\widehat{X}^1$ and $\widehat{X}^2=X^2$, we have $G^{\T}\lambda_2=G^{\T}\lambda_1^1$ in $(C([0,T];\R)^{**})^k$, and hence $G^{\T}\lambda_2\in C([0,T];\R^k)$. Here, $G$ can be viewed a mapping in  $\Lc((C([0,T];\R)^{**})^n;$ $(C([0,T];\R)^{**})^k)$.  By the optimality condition for $\widehat{\alpha}\in\U[0,T]$, we have
\begin{align}\label{SMP_pre}
0&\leq\E\left[\int_0^T\left\langle r_0\nabla_uf(t,\widehat{X}_t,\widehat{\mu}_t,\widehat{\alpha}_t)+\nabla_u b(t,\widehat{X}_t,\widehat{\rho}_t,\widehat{\alpha}_t)^{\T}\lambda_1^1(t)+\nabla_u\sigma(t,\widehat{X}_t,\widehat{\rho}_t,\widehat{\alpha}_t)^{\T}\lambda_1^2(t),\alpha_t-\widehat{\alpha}_t\right\rangle\d t\right.\nonumber\\
&\quad-\left.\sum_{i\in I}\int_0^T\left\langle\nabla_u\phi^i(t,\widehat{X}_t,\widehat{\mu}_t,\widehat{\alpha}_t)\eta_t^i,\alpha_t-\widehat{\alpha}_t\right\rangle\d t\right].
\end{align}
Exploiting the fact that $(\lambda^1_1,\lambda_1^2)$ solves \eqref{BSDE}, we then can derive that \eqref{SMP_pre} indicates \eqref{SMP} by the arbitrariness of $\alpha\in \U[0,T]$. By the optimality condition for $\widehat{\zeta}\in\mathcal{A}^k[0,T]$, we arrive at
\begin{align*}
0&\leq-\E\left[\langle\xi_3,\eta-\widehat{\zeta}\rangle_{(C([0,T];\R)^{**})^n,({\cal M}[0,T])^n}\right]\\
&=\E\left[r_0\int_0^Tc(t)\d(\eta_t-\hat{\zeta}_t)+\left\langle \lambda_2,\int_0^TG(t)\d(\eta_t-\hat{\zeta}_t)\right\rangle_{(C([0,T];\R)^{**})^n,({\cal M}[0,T])^n}\right]\\
&=\E\left[r_0\int_0^Tc(t)\d(\eta_t-\hat{\zeta}_t)+\int_0^TG(t)^{\T}\lambda^1_1(t)\d(\eta_t-\hat{\zeta}_t)\right].
\end{align*}
For the last equality, we have used the fact that $G^{\T}\lambda_2=G^{\T}\lambda_1^1$. Taking $\eta_t=\hat{\zeta}_t+\delta_t\in \mathcal{A}^k[0,T]$ for any $\delta=(\delta_t)_{t\in[0,T]}\in\mathcal{A}^k[0,T]$, we derive that
\begin{align*}
\Pb\left(r_0c_i(t)+G_i(t)^{\T}Y_t\geq 0,~\forall (t,i)\in[0,T]\times\{1,\ldots,k\}\right)=1.  
\end{align*}
Moreover, we take $\eta_t\equiv 0$, this results in \eqref{completeness}.  Thus, the proof of the theorem is complete by just noting that (i) follows from Lemma~\ref{lem:FJCondAbstractprob}.
\end{proof}
		
\begin{remark}
When there are no constraints (i.e. $m=0$) and distribution-dependent terms, Theorem~\ref{SMP_con} together with the constraints qualification (CQ) condition derived in \cite{BWY} will reduce to the SMP  as in \cite{Bahlali}.
\end{remark}
		
\section{Uniqueness and Stability of Solution to Constrained FBSDE}\label{sec:uniqueFBSDE}
		
The aim of this section is to study the uniqueness and stability of solutions to the following generalized type of constrained FBSDE derived in Theorem~\ref{SMP_con}. Without loss of generality, we assume that $d=1$:  
\begin{align}\label{FBSDE}
\widehat{X}_t&=\xi+\int_0^tb(s,\widehat{X}_s,\widehat{\mu}_s,\widehat{\alpha}_s)\d s+\int_0^t\sigma(s,\widehat{X}_s,\widehat{\mu}_s,\widehat{\alpha}_s)\d W_s+\int_0^t G(s)\d\widehat{\zeta}_s,\nonumber\\
Y_t&={r_0\nabla_xg(\widehat{X}_T,\widehat{\mu}_T)+r_0\E'\left[\pa_{\mu}g(\widehat{X}_T',\widehat{\mu}_T)(\widehat{X}_T)\right]}\nonumber\\
&\quad+\int_t^T\nabla_x H^{r_0}(s,\widehat{X}_s,\widehat{\mu},\widehat{\alpha}_s,Y_s,Z_s)\d s+\int_t^T\E'\left[\pa_{\mu}H^{r_0}(s,\widehat{X}_s',\widehat{\mu}_s,\widehat{\alpha}_s',Y_s',Z_s')(\widehat{X}_s)\right]\d s\nonumber\\
&\quad-\int_t^T\nabla_x\phi(s,\widehat{X}_s,\widehat{\mu}_s,\widehat{\alpha}_s)\eta_s\d s-\int_t^T\E'\left[\pa_{\mu}\phi(s,\widehat{X}_s',\widehat{\mu}_s,\widehat{\alpha}_s')(\widehat{X}_s){\eta_s}'\right]\d s\nonumber\\
&\quad-\int_t^TZ_s\d W_s
\end{align}
subject to $|r_0|+\|\eta\|_{\Lb_{\Fb,\R}}=1$,  $\widehat{\mu}_t=\Law(\widehat{X}_t)$, $\Pb\left(G(t)Y_t+{ r_0}c(t)\geq 0,~\forall t\in [0,T]\right)=1$, and $\d t\times\d\Pb$-a.s.
\begin{align}\label{FBSDE2}
\begin{cases}
\displaystyle \left\langle \nabla_u H^{r_0}(t,\widehat{X}_t,\widehat{\mu}_t,\widehat{\alpha}_t,Y_t,Z_t)-\sum_{i\in I}\nabla_u\phi^i(t,\widehat{X}_t,\widehat{\mu}_t,\widehat{\alpha}_t)\eta_t^i, u-\widehat{\alpha}_t\right\rangle\geq 0,~ \forall u\in U,\\[1em]
\displaystyle
\int_0^T \phi(t,\widehat{X}_t,\widehat{\mu}_t,\widehat{\alpha}_t)\eta_t\d t=0,\quad \phi\big(t,\widehat{X}_t,\widehat{\mu}_t,\widehat{\alpha}_t\big)\geq 0,\quad \eta_t\geq 0,\\[1em]
\displaystyle
~\Pb\left(\sum_{i=1}^k\mathbf{1}_{\{r_0c_i(t)+G_i(t)^{\T}Y_t>0\}}\d\widehat{\zeta}_t^i=0\right)=1,
\end{cases}
\end{align}
which characterizes the necessary optimality condition of the constrained MFC problem \eqref{cost_func0} with singular control. To prove the uniqueness of solution $(\widehat{X},\widehat{\alpha},\widehat{\zeta},Y,Z,\eta)=(\widehat{X}_t,\widehat{\alpha}_t,\widehat{\zeta}_t,Y_t,Z_t,\eta_t)_{t\in [0,T]}$ satisfying  \eqref{FBSDE}-\eqref{FBSDE2}, we first give the definition of uniqueness of solutions to Eq.~\eqref{FBSDE}-\eqref{FBSDE2}:
		
\begin{definition}[Uniqueness of solutions to FBSDE~\eqref{FBSDE}-\eqref{FBSDE2}]\label{FBSDE_solution}
Let $(\widehat{X}^i,\widehat{\alpha}^i,\widehat{\zeta}^i,Y^i,Z^i,\eta^i)$ $=(\widehat{X}_t^i,\widehat{\alpha}_t^i,\widehat{\zeta}_t^i,Y_t^i,Z_t^i,\eta_t^i)_{t\in [0,T]}$ for $i=1,2$ be $\Fb$-adapted processes valued in $\R^n\times U\times \R_+^k\times\R^n\times\R^{n\times r}\times \R_+$ satisfying the coupling constrained FBSDE \eqref{FBSDE}-\eqref{FBSDE2}. The solution of FBSDE \eqref{FBSDE}-\eqref{FBSDE2} is said to be unique if $\widehat{X}^1=\widehat{X}^2$, $\widehat{\alpha}^1=\widehat{\alpha}^2$, $Y^1=Y^2$, $Z^1=Z^2$, $\eta^1=\eta^2$  and $\int_0^{\cdot}G(s)\d\widehat{\zeta}_s^1=\int_0^{\cdot}G(s)\d\widehat{\zeta}_s^2$.
\end{definition}
		
\begin{remark}\label{rem:uniqueFNSDE}
In Definition~\ref{FBSDE_solution}, we only require $\int_0^{\cdot}G(s)\d\widehat{\zeta}_s^1=\int_0^{\cdot}G(s)\d\widehat{\zeta}_s^2$ rather than $\widehat{\zeta}^1=\widehat{\zeta}^2$ . Once the existence of optimal triple $(\widehat{X},\widehat{\alpha},\widehat{\zeta})=(\widehat{X}_t,\widehat{\alpha}_t,\widehat{\zeta}_t)_{t\in[0,T]}$ is guaranteed, the FBSDE \eqref{FBSDE}-\eqref{FBSDE2} has a solution with the backward component $(Y,Z,\eta)=(Y_t,Z_t,\eta_t)_{t\in[0,T]}$ being the multipliers given in the last section. Then,  by using Corollary \ref{optimality_exist}, we can conclude that, under Assumption \ref{ass1} and Assumption \ref{ass2}, there exists a solution to the generalized type constrained FBSDE \eqref{FBSDE}-\eqref{FBSDE2}.    
\end{remark}
		
We now verify the {constraint qualification (CQ) condition as in (1.4) of \cite{Zowe} to establish the Karush-Kuhn-Tucker (KKT) optimality condition (i.e., we verify $r_0=1$ in Lemma \ref{lem:FJCondAbstractprob})}. 
\begin{lemma}\label{KKT}
Under Assumption~\ref{ass4}-{\rm(C1)}, the KKT optimality condition holds. More precisely, if $(\widehat{X},\widehat{\alpha},\widehat{\zeta})\in C:=\M^2\times\U[0,T]\times\mathcal{A}^k[0,T]$ is an optimal solution to the constrained MFC problem \eqref{cost_func0}, there exists some $\eta\in \Lb^2_{\Fb,\R+}$ such that 
\begin{itemize}
\item [{\rm(i)}] $\E\left[\int_0^T\phi(t,\widehat{X}_t,\Law(\widehat{X}_t))\eta_t\d t\right]=0$; 
\item[{\rm(ii)}] $\Pb(c_i(t)+G_i(t)^{\T}Y_t\geq 0,~\forall (i,t)\in\{1,\ldots,k\}\times[0,T])=1$, where $c_i(t)$ and $G_i(t)$ are the $i$-th row of $c(t)$ and $G(t)$, respectively. It also holds that
\begin{equation}\label{completeness1}
\Pb\left(\sum_{i=1}^k\mathbf{1}_{\{c_i(t)+G_i(t)^{\T}Y_t>0\}}\d\widehat{\zeta}_t^i=0\right)=1.
\end{equation} 
\end{itemize}
Moreover, the stochastic {\rm(}first-order{\rm)} minimum condition holds that, $\d t\otimes\d \Pb$-a.s.
\begin{align}\label{SMP1}
&\left\langle \nabla_u H(t,\widehat{X}_t,\widehat{\mu}_t,\widehat{\alpha}_t,Y_t,Z_t), u-\widehat{\alpha}_t\right\rangle\geq 0,\quad \forall u\in U.
\end{align}
Here, the Hamiltonian $H(t,x,\mu,u,y,z):=H^{1}(t,x,\mu,u,y,z)$ for $(t,x,\mu,u,y,z)\in[0,T]\times\R^n\times\Pc_2(\R^n)\times\R^l\times\R^n\times\R^{n\times r}$, and the process pair $(Y,Z)=(Y_t,Z_t)_{t\in[0,T]}$ taking values in $\R^n\times\R^{n\times r}$ is the unique solution to the following general type of constrained BSDE:
\begin{align}\label{BSDE}
Y_t&={\nabla_xg(\widehat{X}_T,\widehat{\mu}_T)+\E'\left[\pa_{\mu}g(\widehat{X}_T',\widehat{\mu}_T)(\widehat{X}_T)\right]}\nonumber\\
&\quad+\int_t^T\nabla_x {H}(s,\widehat{X}_s,\widehat{\mu}_s,\widehat{\alpha}_s,Y_s,Z_s)\d s+\int_t^T\E'\left[\pa_{\mu}H(s,\widehat{X}_s',\widehat{\mu}_s,\widehat{\alpha}_s',Y_s',Z_s')(\widehat{X}_s)\right]\d s\nonumber\\
&\quad-\int_t^T\nabla_x\phi(s,\widehat{X}_s,\widehat{\mu}_s)\eta_s\d s-\int_t^T\E'\left[\pa_{\mu}\phi(s,\widehat{X}_s',\widehat{\mu}_s)(\widehat{X}_s)\eta_s'\right]\d s-\int_t^TZ_s\d W_s.
\end{align}
\end{lemma}

\begin{proof}
This is a direct consequence of Lemma 3.9 in \cite{BWY} and Theorem 3.1 of \cite{Zowe}.
\end{proof}
{As a result, we can set $r_0=1$ in \eqref{FBSDE}-\eqref{FBSDE2} under Assumption~\ref{ass4}}. The strict convexity imposed on the Hamiltonian $H$ w.r.t. the control variable $u\in U$ given in (C3) of Assumption~\ref{ass4} yields the result that once $\langle \nabla_u H(t,\widehat{X}_t,\widehat{\mu}_t,\widehat{\alpha}_t,Y_t,Z_t), u-\widehat{\alpha}_t\rangle\geq 0$ for all $u\in U$, we must have $\widehat{\alpha}_t=\hat{u}(t,\widehat{X}_t,\widehat{\mu}_t,Y_t,Z_t)$. Then, we have
\begin{theorem}[Uniqueness of solutions to FBSDE]\label{unique}
Let Assumption~\ref{ass3} and Assumption~\ref{ass4}  hold. The constrained FBSDE \eqref{FBSDE}-\eqref{FBSDE2} has a unique solution in the sense of Definition~\ref{FBSDE_solution}.	
\end{theorem}
		
\begin{proof}
Let $(X^i_t,\zeta^i_t,Y^i_t,Z^i_t,\eta_t^i)_{t\in[0,T]}$ ($i=1,2$) be two solutions to the FBSDE \eqref{FBSDE}-\eqref{FBSDE2} satisfying $X^1_0=X_0^2$ and $Y_T^1=Y_T^2$. { It follows from Lemma~\ref{KKT} that the KKT condition holds under Assumption~\ref{ass3}, and hence we can set the coefficient $r_0=1$ in the Hamiltonian $H^{r_0}$. Recall $F(t,x,\mu,u,y,z):=\nabla_x H(t,x,\mu,u,y,z)+\int_{\R^n}\pa_{\mu}H(t,x',\mu,u,y,z)(x)\mu(\d x')$ defined in Assumption~\ref{ass4}-(C3).} Let us set $\bar{X}_t:=X_t^1-X_t^2$, $\bar{\zeta}_t:=\zeta_t^1-\zeta_t^2$, $\bar{Y}_t:=Y_t^1-Y_t^2$, $\bar{Z}_t:=Z_t^1-Z_t^2$ and $\bar{\eta}_t:=\eta_t^1-\eta_t^2$ for $t\in[0,T]$. Applying It\^{o}'s formula to $\bar{X}_t\bar{Y}_t$, we arrive at
{\begin{align*}
\bar{X}_T\bar{Y}_T&=\int_0^T\left\{-\Delta F(t)\bar{X}_t+\bar{Y}\Delta{b}(t)+\tr(\bar{Z}_t\Delta{\sigma}(t))\right\}\d t+\int_0^T\left(\bar{X}_t\bar{Z}_t+\bar{Y}_t\Delta\sigma(t)\right)\d W_t\\
&\quad+\int_0^T\left(A(t)\bar{\eta}_t+B(t)\E[\bar{\eta}_t]\right)\bar{X}_t\d t+\int_0^TG(t)\bar{Y}_t\d\bar{\zeta}_t,
\end{align*}}where the coefficients
{$\Delta\Psi=\Psi(X_T^1,\Law(X_T^1))-\Psi(X_T^2,\Law(X_T^2))$, $\Delta F(t):=F(t,X_t^1,\Law(X_t^1),Y_t^1,Z_t^1)-F(t,X_t^2,\Law(X_t^2),Y_t^2,Z_t^2)$, $\Delta b(t):=b(t,X_t^1,\Law(X_t^1),Y_t^1,Z_t^1)-b(t,X_t^2,\Law(X_t^2),Y_t^2,Z_t^2)$ and  $\Delta\sigma(t):=\sigma(t,X_t^1,\Law(X_t^1),Y_t^1,Z_t^1)-\sigma(t,X_t^2,\Law(X_t^2),Y_t^2,Z_t^2)$ }

Taking expectations on both sides and noticing {that} $\bar{Y}_T=\Delta\Psi$, we get
{\small\begin{align*}
\E\left[\bar X_T\Delta\Psi\right]&=\E\left[\int_0^T\left\{(\Phi(t,X_t^1,\Law(X_t^1),Y_t^1,Z_t^1)-\Phi(t,X_t^2,\Law(X_t^1),Y_t^2,Z_t^2))(X_t^1-X_t^2,Y_t^1-Y_t^2,Z_t^1-Z_t^2)\right.\right.\\
&\qquad~+\left.(\Phi(t,X_t^2,\Law(X_t^1),Y_t^2,Z_t^2)-\Phi(t,X_t^2,\Law(X_t^2),Y_t^2,Z_t^2))(X_t^1-X_t^2,Y_t^1-Y_t^2,Z_t^1-Z_t^2)\right\}\d t\\
&\qquad~+\left.\int_0^T\bar{\eta}_t\left\{a(t)\bar{X}_t+b(t)\E[\bar{X}_t]\right\}\d t+\int_0^TG(t)\bar{Y}_t\d\bar{\zeta}_t\right].
\end{align*}}Here, we applied Fubini's theorem in the last line. Note that, it holds that
\begin{align*}
&\E\left[\int_0^T\bar{\eta}_t\left\{A(t)\bar{X}_t+B(t)\E[\bar{X}_t]\right\}\d t\right]=\E\left[\int_0^T\bar{\eta}_t\left\{\phi(t,X_t^1,\Law(X_t^1))-\phi(t,X_t^2,\Law(X_t^2))\right\}\d t\right]\\
&\qquad\qquad=-\E\left[\int_0^T\left\{\phi(t,X_t^1,\Law(X_t^1))\eta_t^2+\phi(t,X_t^2,\Law(X_t^2))\eta_t^1\right\}\d t \right]\leq 0,
\end{align*}
where, we used the fact that $\E[\int_0^T\phi(t,X_t^i,\Law(X_t^i))\eta_t^i\d t]=0$, $\phi_i(t,X_t^i,\Law(X_t^i))\geq 0$ and $\eta_t^i\geq 0$ for $i=1,2$. Similarly, we have $\E[\int_0^TG(t)\bar{Y}_t\d\bar\zeta_t]\leq0$. Exploiting Assumption~\ref{ass4}, we deduce that 
\begin{align*}
0&\leq \E\left[\bar X_T\Delta\Psi\right]\leq\E\big[\left(\Phi(t,X_t^2,\Law(X_t^1),Y_t^2,Z_t^2)-\Phi(t,X_t^2,\Law(X_t^2),Y_t^2,Z_t^2)\right)\left(\bar{X}_t,\bar{Y}_t,\bar{Z}_t\right)\nonumber\\
&\qquad\qquad\qquad\qquad\quad-\beta\left(|\bar{X}_t|^2+|\bar{Y}_t|^2+|\bar{Z}_t|^2\right)\big].
			\end{align*}
However, it follows from Assumption~\ref{ass1}-(A3) that
\begin{align*}
&\left(\Phi(t,X_t^2,\Law(X_t^1),Y_t^2,Z_t^2)-\Phi(t,X_t^2,\Law(X_t^2),Y_t^2,Z_t^2)\right)(\bar{X}_t,\bar{Y}_t,\bar{Z}_t)\\
&\qquad\leq \frac{M}{2}\mathcal{W}_2(\Law(X_t^1),\Law(X_t^2))\left(|X_t^1-X_t^2|+|Y_t^1-Y_t^2|+|Z_t^1-Z_t^2|\right)\\
&\qquad\leq \frac{M}{2}\left(|X_t^1-X_t^2|^2+|Y_t^1-Y_t^2|^2+|Z_t^1-Z_t^2|^2\right).
\end{align*}
Hence, combining the above two inequalities leads to $\bar{X}=0,\bar{Y}=0$ and $\bar{Z}=0$. It then follows that $\bar{\alpha}=0$ and  $\int_0^{\cdot}G(s)\d\bar{\zeta}_s=0$. Moreover, we have $\int_0^tA(s)\bar{\eta}_s\d s+\int_0^tB(s)\E\left[\bar{\eta}_s\right]\d s=0$ for all $t\in[0,T]$. Taking expectations on both sides of the above result, together with (C1) in Assumption~\ref{ass4}, we conclude that $\E\left[\bar{\eta}_t\right]=0$, for $m$-a.e. $t\in [0,T]$, and hence $\int_0^tA(s)\bar{\eta}_s\d s=0$, for $m$-a.e.~$t\in[0,T]$. This indicates that $\bar{\eta}=0$ in $\Lb_{\Fb}^{2,1}$ by utilizing Assumption~\ref{ass4}-(C1). 
\end{proof}
	Finally, we examine the stability of the solution to the constrained FBSDE \eqref{FBSDE}-\eqref{FBSDE2}. To this end, we introduce the following enhanced conditions:
		{\it\begin{itemize}
			\item [\rm (C1')]
				The constrained function $\phi$ is independent of the control variable and admits the linear form that $\phi(t,x,\mu)=A(t)x+B(t)\int_{\R^n}x\mu(\d x)+C(t)$ for $(t,x,\mu)\in[0,T]\times\R^n\times\Pc_2(\R^n)$. Here, the coefficients {$(A, B)\in C([0,T],\R^n\times\R^n)$ and $C\in C([0,T];\R)$ satisfying $|A(t)|\geq c_0>0$ and $|A(t)+B(t)|\geq c_0>0$ for some $c_0>0$ independent of $t\in[0,T]$. }
            \item [\rm (C5')] Set $\Psi(x,\mu):=\nabla_x g(x,\mu)+\int_{\R^n}\pa_{\mu} g(y,\mu)(x)\mu(\d y)$ for $(x,\mu)\in\R^n\times\Pc_2(\R^n)$. The function $\Psi$ is Lipschitz continuous and satisfies the coercivity condition that, for some $\alpha\geq0$ and $(x^1,\mu^1,X^1),(x^2,\mu^2,X^2)\in \R^n\times\Pc_2(\R^n)\times L^2((\Omega,\F_T,\Pb),\R^n)$,
\begin{align*}
\E\left[\left\langle\Psi(X^1,\Law(X^1))-\Psi(X^2,\Law(X^2)),X^1-X^2\right\rangle\right]\geq \alpha\E[|X^1-X^2|^2].
\end{align*}
		\end{itemize}}
		Then, we have
		\begin{theorem}\label{stable}
			Let Assumption~\ref{ass3}, {\rm(C1')}, {\rm(C3)-(C4)} of Assumption~\ref{ass4} hold and  
			$\xi_1,\xi_2\in L^2((\Omega,\F_0,\Pb);\R^n)$. {Moreover, suppose that  {\rm (C5')} holds for either $\Psi^1$ or $\Psi^2$.} Denote by $(\widehat{X}^i,\widehat{\alpha}_t^i,\widehat{\zeta}^i_t,Y^i,Z^i,\eta^i)$ the unique solutions to the FBSDE \eqref{FBSDE}-\eqref{FBSDE2} with respective initial data $\xi_i$. Set $\bar{X}_t:=X_t^1-X_t^2$, $\bar{\zeta}_t:=\zeta_t^1-\zeta_t^2$, $\bar{Y}_t:=Y_t^1-Y_t^2$, $\bar{Z}_t:=Z_t^1-Z_t^2$ and  $\bar{\eta}_t:=\eta_t^1-\eta_t^2$ for $t\in[0,T]$. Then, there exists a constant $C>0$ depending on $(T,M,\alpha,\beta)$ only such that{
\begin{align}
&\E\left[\left|\bar{X}_T\right|^2+\int_0^T\left(|\bar{X}_t|^2+|\bar{\alpha}_t|^2+|\bar{Y}_t|^2+|\bar{Z}_t|^2+|\bar{\eta}_t|^2\right)\d t+\left|\int_0^TG(t)\d\bar{\zeta}_t\right|^2\right]\nonumber\\
&\qquad\qquad\leq C\left\{\E\left[|\xi_1-\xi_2|^2\right]+\sup_{(x,\mu)\in\R^n\times\Pc_2(\R^n)}\left|\Psi^1(x,\mu)-\Psi^2(x,\mu)\right|^2\right\}\label{stability}
\end{align}
with $\Psi^i(x,\mu):=\nabla_x g^i(x,\mu)+\int_{\R^n}\pa_{\mu} g^i(y,\mu)(x)\mu(\d y)$ for $i=1,2$.}
			
		\end{theorem}
		
\begin{proof}
Recall $\Delta b(t)$, $\Delta\sigma(t)$ and $\Delta F(t)$ defined in the proof of Theorem~\ref{unique}. Then, one has 
\begin{align}\label{Gzeta}
				\int_0^tG(s)\d\bar{\zeta}_s=\bar{X}_t-(\xi_1-\xi_2)-\int_0^t\Delta b(s)\d s-\int_0^t\Delta\sigma(s){ \d W_s}.
			\end{align}
			On the other hand, applying It\^{o}'s formula to $|\bar{X}_t|^2$ leads to
			\begin{align}\label{Xsquare}
				\d |\bar{X}_t|^2=2\left(\Delta b(t)^{\T}\bar{X}_t+\tr\left(\Delta\sigma\Delta\sigma^{\T}(t)\right)\right)\d t+2\bar{X}_t^{\T}\Delta\sigma(t)\d W_t+2\bar{X}_t^{\T}G(t)\d\bar{\zeta}_t.
			\end{align}
			By the Young's inequality, $\left|\int_0^t2\bar{X}_s^{\T}G(s)\d\bar{\zeta}_s\right|\leq\epsilon\sup_{0\leq s\leq t}|\bar{X}_s|^2+\frac{1}{\epsilon}\left|\int_0^t G(s)\d\bar{\zeta}_s \right|^2$ for any $\epsilon\in(0,1)$. Plugging the above inequality into the equality \eqref{Xsquare}, and taking expectations on both sides, we arrive at, for any $t\in[0,T]$,
			\begin{align*}
			&	\E\left[\sup_{0\leq s\leq t}|\bar{X}_s|^2\right]\leq\E\left[\sup_{0\leq s\leq t}\int_0^s 2\left(\Delta b(r)^{\T}\bar{X}_r+\tr\left(\Delta\sigma\Delta\sigma^{\T}(r)\right)\right)\d r\right]\nonumber\\
            &\qquad+\E\left[\sup_{0\leq s\leq t}\int_0^s 2\bar{X}_r^{\T}\Delta\sigma(r)\d W_r\right]+\epsilon\E\left[\sup_{0\leq s\leq t}|\bar{X}_s|^2\right]+\frac{1}{\epsilon}\E\left[\left|\int_0^t G(s)\d\bar{\zeta}_s \right|^2\right].
			\end{align*}
			Using (C4) of Assumption~\ref{ass4}, the equality \eqref{Gzeta}, the BDG's inequality and the Gronwall's inequality, we deduce the existence of a constant $C>0$ depending on $(M,\epsilon, T)$ only such that
			\begin{align}\label{Xsup}
				\E\left[\sup_{0\leq t\leq T}|\bar{X}_t|^2\right]\leq C\E\left[|\xi^1-\xi^2|^2+\int_0^{T}(|\bar{Y}_t|^2+|\bar{Z}_t|^2)\d t\right].
			\end{align}
			With the help of the BSDE in \eqref{FBSDE}, we have, for all $t\in [0,T]$,
			\begin{align}\label{eta}
				\int_t^TA(s)\bar{\eta}_s\d s+\int_t^TB(s)\E\left[\bar{\eta}_s\right]\d s=\bar{Y}_T-Y_t+\int_t^T\Delta F(s)\d s-\int_t^T\bar{Z}_s\d W_s.
			\end{align}
			Applying the It\^{o}'s formula to $|\bar{Y}_t|^2$, we deduce, for all $t\in [0,T]$,
			\begin{align*}
				\left|\bar{Y}_t\right|^2+\int_t^T|\bar{Z}_s|^2\d s&={|\bar Y_T|^2}+2\int_t^T\bar{Y}_s^{\T}\Delta F(s)\d s-2\int_t^T\bar{Y}_s\bar{Z}_s\d W_s-2\int_t^T\bar{\eta}_sA(s)^{\T}\bar{Y}_s\d s\nonumber\\
				&\quad-2\int_t^T\E\left[\bar{\eta}_s\right]B(s)^{\T}\bar{Y}_s\d s.
			\end{align*}
			By using the Young's inequality again, for all $t\in [0,T]$,
			\begin{align*}
				\left|2\int_t^T\bar{\eta}_sA(s)^{\T}\bar{Y}_s\d s\right|&\leq \epsilon\sup_{t\leq s\leq T}|\bar{Y}_s|^2+\frac{1}{\epsilon}\left|\int_t^TA(s)^{\T}\bar{\eta}_s\d s\right|^2,\\
				\left|2\int_t^T\E\left[\bar{\eta}_s\right]B(s)^{\T}\bar{Y}_s\d s\right|&\leq \epsilon\sup_{t\leq s\leq T}|\bar{Y}_s|^2+\frac{1}{\epsilon}\left|\int_t^TB(s)^{\T}\E\left[\bar{\eta}_s\right]\d s\right|^2.
			\end{align*}
			Using Assumption~\ref{ass4}-(C4), the condition (C1') (C5'), the BDG's inequality and the Gronwall's inequality, we can similarly derive, for some constant $C>0$ depending on $(M,\epsilon,T)$ only,
			\begin{align}\label{Ysup}
				&\E\left[\sup_{0\leq t\leq T}|\bar{Y}_t|^2+\int_0^T|\bar{Z}_t|^2\d t \right]\leq C\E\left[|\bar{X}_T|^2+\int_0^T|\bar{X}_t|^2\d t\right].
			\end{align}
			Applying integration by parts to $\bar{X}_t\bar{Y}_t$ from $t=0$ to $T$, we have
			\begin{align*}
				\bar{X}_T\bar{Y}_T&=\int_0^T(-\Delta \bar{F}(t)\bar{X}_t+\bar{Y}\Delta\bar{b}(t)+\tr(\bar{Z}_t\Delta\bar{\sigma}(t)))\d t+\int_0^T(\bar{X}_t\bar{Z}_t+\bar{Y}_t\Delta\sigma)\d W_t\\
				&\quad+\int_0^T\left\{A(t)\bar{\eta}_t+B(t)\E[\bar{\eta}_t]\right\}\bar{X}_t+\int_0^TG(t)\bar{Y}_t\d\bar{\zeta}_t.
			\end{align*}
Taking expectations on both sides of the above equality and utilizing {\rm (C5')} (we assume (C5') holds for $\Psi^1$ without loss of generality), we deduce that
{\begin{align}\label{XY}
&\alpha\E\left[|\bar X_T|^2\right]+\beta\E\left[\int_0^T(|\bar{X}_t|^2+|\bar{Y}_t|^2+|\bar{Z}_t|^2)\d t\right]\nonumber\\
&\quad\leq \E\left[\bar{Y}_0^{\T}\left(\xi^1-\xi^2\right)\right]+\sup_{(x,\mu)\in\R^n\times\Pc_2(\R^n)}\left|\Psi^1(x,\mu)-\Psi^2(x,\mu)\right|^2\nonumber\\
&\quad\leq \frac{2}{\epsilon}\E\left[\left|\bar{Y}_0\right|^2\right]+\frac{1}{2\epsilon}\E\left[\left|\xi^1-\xi^2\right|^2\right]+\sup_{(x,\mu)\in\R^n\times\Pc_2(\R^n)}\left|\Psi^1(x,\mu)-\Psi^2(x,\mu)\right|^2.
\end{align}}
Choosing $\epsilon>0$ small enough and taking into account of \eqref{Xsup} and \eqref{Ysup}, we can find a constant $C>0$ depending on $M,T,\alpha,\beta$ only such that
\begin{align}\label{XYZ}
&\E\left[\left|\bar{X}_T\right|^2+\int_0^T\left(|\bar{X}_t|^2+|\bar{Y}_t|^2+|\bar{Z}_t|^2\right)\d t\right]\nonumber\\
&\quad\leq C\left(\E\left[\left|\xi^1-\xi^2\right|^2\right]+\sup_{(x,\mu)\in\R^n\times\Pc_2(\R^n)}\left|\Psi^1(x,\mu)-\Psi^2(x,\mu)\right|^2\right).
\end{align}
Combining the Lipschitz continuity of $\hat{u}:[0,T]\times\R^n\times\Pc_2(\R^n)\times\R^n\times\R^{n\times r}\mapsto U\subset\R^l$ in Assumption~\ref{ass4}-(C3), \eqref{Gzeta}, \eqref{eta} and the condition (C1'), we conclude the stability \eqref{stability} from \eqref{XYZ} above.
\end{proof}
		
{
\begin{remark}
We recall that the existence of an optimal strict control is proved in Section~\ref{sec:well-posedRelax} (see Theorem~\ref{existence_opt} and Corollary~\ref{optimality_exist}). The necessary conditions characterizing the optimal control are given in Theorem~\ref{SMP_con} using FJ condition, which in turn ensures the existence of solution to the associated FBSDE. Moreover, the uniqueness of the solution to the constrained FBSDE is established in Section~\ref{sec:uniqueFBSDE}. In sum, the combination of Theorem~\ref{existence_opt}, the characterization results in Section~\ref{sec:SMP} and the uniqueness result in Section~\ref{sec:uniqueFBSDE} provide a complete theory for the existence and uniqueness of solution to the associated FBSDE and the characterization of the optimal control for the general constrained MFC problem. 
\end{remark}
}

{Finally, we present an application example in  monetary reserve management of interbank system similar to \cite{CarmonaFouqueSun15} and \cite{BLY22}. Consider the log-monetary reserve process of a representative bank within the interbank system under equilibrium governed by the following controlled McKean-Vlasov SDE:
\begin{align}\label{eq:statebank}
dX_t^{\zeta}=\alpha(\mathbb{E}[X_t^{\zeta}]-X_t^{\zeta})dt + \sigma dW_t + d\zeta_t,\quad t\in[0,T].
\end{align}
Here, the term $\alpha(\mathbb{E}[X_t^{\zeta}]-X_t^{\zeta})dt$ with $\alpha>0$ depicts that the bank’s reserves are constantly adjusting toward the average reserve level of the banking system. The term  $\sigma dW_t$ with $\sigma>0$ describes the exogenous noise or randomness, representing market uncertainty affecting the bank’s liquidity, and the term $d\zeta_t$ stands for the singular control that captures the active intervention of the bank with the central bank. In particular, if $\zeta=(\zeta_t)_{t\in[0,T]}$ is  absolutely continuous w.r.t. $t$, then $\dot{\zeta}_t:=d\zeta_t/dt$ represents the rate of borrowing/lending to the central bank. 

The goal of this MFC problem is to choose an optimal control $\zeta = (\zeta_t)_{t\in[0,T]}$ that minimizes the cost functional that
\begin{align}\label{eq:objective}
\inf_{\zeta\in\mathscr{U}_{\text{ad}}[0,T]}J(\zeta)= \inf_{\zeta\in\mathscr{U}_{\text{ad}}[0,T]}\mathbb{E}\left[\int_0^T \beta\left(\mathbb{E}[X_t^{\zeta}]-X_t^{\zeta}\right)^2dt + \int_{[0,T]} \gamma d\zeta_t\right].
\end{align}
The objective functional includes two costs: (i) the deviation cost: the term $\beta(\mathbb{E}[X_t^{\zeta}]-X_t^{\zeta})^2$ penalizes the squared deviation of the individual bank’s reserves from the system average. A higher $\beta>0$ implies a stronger penalty for idiosyncratic risk or deviation from market norms, pushing the bank toward the mean; (ii) the trading cost: the term $\int_0^T \gamma d\zeta_t$ accounts for the cost of executing trades with the central bank. The parameter $\gamma\in\mathbb{R}$ determines whether borrowing from or lending to the central bank incurs a cost or generates profit.

The system is subject to the following dynamical constraint, which imposes a bound on the deviation between the expected reserve and the actual reserve:
\begin{align}\label{eq:dynamicalconstraint}
\left|\kappa\mathbb{E}[X_t^{\zeta}]-X_t^{\zeta}\right|\leq c,\quad dt\times d\mathbb{P}\mbox{-}a.s.,
\end{align}
where $c > 0$ serves as a predetermined threshold value and the term  $\kappa\mathbb{E}[X_t^{\zeta}]$ represents a scaled average of log‑monetary reserve with scaling factor $\kappa\in(0,1)$. The form of constraint in \eqref{eq:dynamicalconstraint} reflects a tolerance buffer: a bank is allowed to deviate from the system average, but only up to a proportion $(1-\kappa)$ of that average. In practice, $\kappa$ can be interpreted as a regulatory confidence factor or a liquidity-adjustment coefficient. A lower $\kappa$ permits larger individual deviations from the mean, offering more flexibility.

From the model described by \eqref{eq:statebank}, \eqref{eq:objective} and \eqref{eq:dynamicalconstraint}, for $(x,\mu)\in\R\times{\cal P}_2(\R)$, it holds that
\begin{align*}
    b(x,\mu)=\alpha\left(\int_{\mathbb{R}} y\mu(dy)-x\right),~~\sigma(x,\mu)\equiv\sigma,~~G(t)\equiv1,~~f(x,\mu)=\beta\left(\int_{\mathbb{R}} y\mu(dy)-x\right)^2,
\end{align*}
while $c(t)\equiv\gamma$ and the constraint functions are $\phi^1(x,\mu)=\kappa\int_{\mathbb{R}} y\mu(dy)-x+c$ and $\phi^2(x,\mu)=x-\kappa\int_{\mathbb{R}} y\mu(dy)+c$ in terms of dynamical constraint \eqref{eq:dynamicalconstraint}. It is not difficult to verify that the above model coefficients satisfy Assumption\ref{ass1} and Assumption \ref{ass3}. As the constraint functions $\phi^1(x,\mu)$ and $\phi^2(x,\mu)$ are independent of regular controls and satisfy Assumption \ref{ass4}-(C1), the KKT condition holds by Lemma \ref{KKT}. The Hamiltonian in Theorem \ref{SMP_con} becomes that, for $(x,\mu,y,z)\in\R\times{\cal P}_2(\R)\times\R^2$,
\begin{align*}
 H(x,\mu,y,z)=\alpha y\left(\int_{\mathbb{R}} y\mu(dy)-x\right)+\sigma z + \beta\left(\int_{\mathbb{R}} y\mu(dy)-x\right)^2.   
\end{align*}
Let $\widehat{\zeta}=(\widehat{\zeta}_t)_{t\in [0,T]}$ be the optimal singular control and $\widehat{X}=(\widehat{X}_t)_{t\in[0,T]}$ be the state dynamics in \eqref{eq:statebank} under $\widehat{\zeta}$. The SMP established in Theorem \ref{SMP_con} gives that, the optimal singular control $\widehat{\zeta}=(\widehat{\zeta}_t)_{t\in [0,T]}$ of the above C-MFC problem can be characterized by
\begin{align}\label{eq:singularcontrol-cond}
    \E\left[\int_0^T\boldsymbol{1}_{\{Y_t+\gamma>0\}}\d\widehat{\zeta}_t\right]=0,
\end{align}
where the pair $(Y,Z)=(Y_t,Z_t)_{t\in[0,T]}$ satisfies $Y_t\geq-\gamma$ for all $t\in[0,T]$, $\mathbb{P}$-a.s., and moreover, the pair of processes solves the constrained BSDE:
\begin{align}\label{eq:examBSDE}
\d Y_t&=\alpha(Y_t-\E[Y_t])\d t-2\beta(\widehat{X}_t-\E[\widehat{X}_t])\d t+(\widehat{\eta}_t^1-\kappa\E[\widehat{\eta}_t^1])\d t-(\widehat{\eta}_t^2-\kappa\E[\widehat{\eta}_t^2])\d t+Z_tdW_t,\nonumber\\
Y_T&=0.
\end{align}
{For this mean-field linear BSDE, the well-posedness can be addressed by Theorem \ref{unique}.} Additionally, the non-negative adapted square-integrable processes  $\widehat{\eta}^i=(\widehat{\eta}_t^i)_{t\in [0,T]}$ for $i=1,2$ satisfy that  
\begin{align}\label{eq:consetaexam}
\E\left[\int_0^T\left(\kappa\E[\widehat{X}_t]-\widehat{X}_t+c\right)\widehat{\eta}_t^1\d t\right]=\E\left[\int_0^T\left(\widehat{X}_t-\kappa\E[\widehat{X}_t]+c\right)\widehat{\eta}_t^2\d t\right]=0.
\end{align}}

{
\begin{remark}\label{rem:meanfieldapprox}

Our C-MFC problem may be explained as the approximation of large-scale finite $N$-player cooperative game. An interesting future research is to establish the connection between the C-MFC problem and the N-player cooperative game under dynamical constraints. However, the main difficulty lies in the complex dynamic constraints in the N-player controlled system. In fact, to establish the connection of the approximation, one possible way is to transform the constrained N-player cooperative game into a problem without constraints. Then, we may apply the Lagrange-multiplier approach in \cite{BWY} to derive the SMP for the $N$-player cooperative game problem. We then will encounter a system of coupled FBSDEs, which is in general very challenging. Moreover, the approximation or the propagation of chaos will further require the joint convergence of the associated Lagrange multipliers and singular controls. One may also need to verify that the limiting model still satisfies the dynamic constraints, which introduces additional analytical difficulties. We leave this interesting but challenging problem for future study.
\end{remark}}
\noindent\textbf{Acknowledgements}  The authors sincerely thank three anonymous referees for their constructive comments and suggestions, which significantly improve the manuscript. L. Bo is supported by National Natural Science Foundation of China (No. 12471451), Natural Science Basic Research Program of Shaanxi (No. 2023-JC-JQ-05), Shaanxi Fundamental Science Research Project for Mathematics and Physics (No. 23JSZ010). X. Yu and J. Wang are supported by the Hong Kong RGC General Research Fund (GRF) under grant No. 15214125 and by the Hong Kong Polytechnic University research grant under No. P0045654.   
		
\begin{appendix}
\section*{Appendix: Proofs of Auxiliary Results}\label{appn} 
			
This appendix collects proofs of the auxiliary results presented in the main body of the paper.
			
\begin{proof}[Proof of Lemma~\ref{strictcontrol}]
For any $\omega\in\Omega$, define the set $K(t,\omega):=K(t,X_t(\omega),\Law^R(X_t^{\alpha,\zeta}))$ with $(t,\omega)\in[0,T]\times\Omega$. Since $K(t,\omega)$ is convex and closed, it follows that
\begin{align*}
\bigg(&\int_Ub(t,X_t(\omega),\Law^R(X_t),u)q_t(\omega,\d u),\int_U\sigma\sigma^{\T}(t,X_t(\omega),\Law^R(X_t),u)q_t(\omega,\d u),\\
&\quad\int_Uf(t,X_t(\omega),\Law^R(X_t),u)q_t(\omega,\d u),\int_U\phi^i(t,X_t(\omega),\Law^R(X_t),u)q_t(\omega,\d u),~i\in I\bigg)\in K(t,\omega)
\end{align*}
is equivalent to 
\begin{align*}
U(t,\omega):=\Big\{u\in& U;~\int_Ub(t,X_t(\omega),\Law^R(X_t),u)q_t(\omega,\d u)=b(t,X_t(\omega),\Law^R(X_t),u),\\
&\int_U\sigma\sigma^{\T}(t,X_t(\omega),\Law^R(X_t),u)q_t(\omega,\d u)=\sigma\sigma^{\T}(t,X_t(\omega),\Law^R(X_t),u),\\
&\int_Uf(t,X_t(\omega),\Law^R(X_t),u)q_t(\omega,\d u)\geq f(t,X_t(\omega),\Law^R(X_t),u),\\
&\int_U\phi^i(t,X_t(\omega),\Law^R(X_t),u)q_t(\omega,\d u)\leq\phi^i(t,X_t(\omega),\Law^R(X_t),u),~i\in I
\Big\}\neq\varnothing.
\end{align*}
For $N\geq1$, define $U^N(t,\omega):=\{u\in U(t,\omega);~|u|\leq N\}$. Then, for each $(t,\omega)\in [0,T]\times\Omega$, $U^N(t,\omega)$ increases to $U(t,\omega)$ as $N$ increases. We also introduce $B^N:=\{(t,\omega)\in[0,T]\times\Omega;~U^N(t,\omega)\neq\varnothing\}$. Since $U^{{N}}(t,\omega)$ is bounded and closed, it is also a compact subset of $U\subset\R^l$. Then, there exists a measurable selection $\alpha^N:B^N\to U$ such that $\alpha^N_t(\omega)\in U^N(t,\omega)$ for $(t,\omega)\in [0,T]\times\Omega$ and $\alpha^N$ is $\mathcal{B}([0,t])\otimes\mathcal{F}_t/\mathcal{B}(\R^l)$ measurable due to \cite{Leese}. Hence, we can define an $\Fb$-progressively measurable process $\alpha=(\alpha_t)_{t\in[0,T]}$ by $\alpha_t(\omega)=\sum_{N=1}^{\infty}\alpha^N_t(\omega)\mathbf{1}_{B_N\backslash B_{N-1}}(t,\omega)$.
				
Next, we show that $\alpha=(\alpha_t)_{t\in[0,T]}$ is an admissible (strict) control. Using the growth condition imposed on $f$ in Assumption~\ref{ass1}-(A6), we have from the above construction of $\alpha$ that
\begin{align*}
+\infty&>\E^R\left[\int_0^T\int_Uf(t,X_t,\Law^R(X_t),u)q_t(\d u)\d t\right]\geq \E^R\left[\int_0^Tf(t,X_t,\Law^R(X_t),\alpha_t)\d t\right]\\
&\geq -\int_0^TM(1+\E^R[|X_t|^p]+M_2(R^X(t)))\d t+M_1\E^R\left[\int_0^T|\alpha_t|^p\d t\right].
\end{align*}
This, together with Lemma~\ref{moment_p}, indicates that $\E^R[\int_0^T|\alpha_t|^p\d t]<\infty$. On the other hand, using the construction of $\alpha$ again, we deduce that, for any $\psi\in C_b^2(\R^n)$ and $(t,\omega)\in[0,T]\times\Omega$, 
\begin{align*}
& \psi(X_t(\omega))-\int_0^t\mathbb{G}\psi(s,X_s(\omega),\mathcal{L}^R(X_s),\alpha_s(\omega))\d s-\int_0^t\nabla_x\psi(X_s(\omega))G(s)\d\zeta_s(\omega)\\
&\qquad-\sum_{0<s\leq t}\left(\psi(X_s(\omega))-\psi(X_{s-}(\omega))-\nabla_x\psi(X_{s-}(\omega))\Delta X_s(\omega)\right)=M_t^R\psi(\omega)
\end{align*} 
is an $(R,\Fb)$-martingale as $R$ is a control rule in Definition~\ref{admissible}. Moreover, we have, for $m\times R$-$\as~(t,\omega)\in[0,T]\times\Omega$,
\begin{align*}
\phi^i\left(t,X_t(\omega),\Law^R(X_t),\alpha_t(\omega)\right)\geq\int_U\phi^i(t,X_t(\omega),\Law^R(X_t),u)q_t(\omega,\d u)\geq0,~~\forall i\in I.    
\end{align*}
This yields that $\alpha=(\alpha_t)_{t\in[0,T]}$ is admissible. Furthermore, by the construction of $\alpha$, it holds that 
\begin{align*}
J(\alpha,\zeta)&=\E^R\left[\int_0^Tf(t,X_t,\Law^R(X_t),\alpha_t)\d t+{\int_{[0,T]} c(t)\d\zeta_t+g(X_T,\Law^R(X_T))}\right]\\
&\leq\E^R\left[\int_0^T\int_Uf(t,X_t,R^X(t),u)q_t(\d u)\d t+{\int_{[0,T]} c(t)\d\zeta_t+g(X_T,\Law^R(X_T))}\right]=\Gamma(R).
\end{align*}
Thus, we complete the proof of the lemma.  
\end{proof}

\begin{proof}[Proof of Lemma~\ref{closedset}]
For $R\in\Pc_2(\Omega;\widehat{\Omega})$ and $i\in I$, let us define the following subset $\tilde{A}_i^R$ of $\Omega$ by
\begin{align*}
\tilde{A}_i^R:=\left\{\omega\in\Omega;~\int_{t_1}^{t_2}\int_U\phi^i(t,X_t(\omega),R^X(t),u)q_t(\omega,\d u)\d t\geq 0,~\forall 0\leq t_1<t_2\leq T\right\}.  
\end{align*}
We first show that $(m\times R)(A_i^R)=T$ iff $R(\tilde{A}_i^R)=1$. Here, recall that $A_i^R\subset[0,T]\times\Omega$ is defined by \eqref{eq:AiR}. 
				
\textit{Necessity}. For any $\omega\in \Omega$, we define the $\omega$-section of $A_i^R$ that
\begin{align*}
A_i^R(\omega)=\left\{t\in[0,T];~\int_U\phi^i(t,X_t(\omega),R^X(t),u)q_t(\omega,\d u)\geq 0 \right\}.
\end{align*}
Then, it follows from Fubini's Theorem that $T=(m\times R)(A_i^R)=\int_{\Omega}m(A_i^R(\omega))R(\d \omega)$. This implies that, there exists some $R$-null set $N\in\F$ such that $m(A_i^R(\omega))=T$ for all $\omega\in\Omega\backslash N$. For such $\omega$ and $0\leq t_1<t_2\leq T$, it holds that
\begin{align*}
&\int_{t_1}^{t_2}\int_U\phi^i(t,X_t(\omega),R^X(t),u)q_t(\omega,\d u)\d t\nonumber\\
&\qquad=\int_{(t_1,t_2]\cap A_i^R(\omega)}\int_U\phi^i(t,X_t(\omega),R^X(t),u)q_t(\omega,\d u)\d t\geq 0.
\end{align*}
Consequently, $\omega\in\tilde{A}_i^R$, or more precisely,  $\Omega\backslash N\subset \tilde{A}_i^R$, and hence $R(\tilde{A}_i^R)=1$.
				
\textit{Sufficiency}. For $\omega\in \tilde{A}_i^R$, we have $m(A_i^R(\omega))=T$, and Fubini's Theorem results in $(m\times R)(A_i^R)=\int_{\Omega}m(A_i^R(\omega))R(\d \omega)=T$. Now, we show that $\tilde{A}_i^R$ is closed for any $R\in\Pc_2(\Omega;\widehat{\Omega})$. Let $\omega^{\ell}=(\psi^{\ell},q^{\ell},\zeta^{\ell})\to\omega=(\psi,q,\zeta)$ in $\Omega$ as $\ell\to\infty$ and $\omega^{\ell}\in\tilde{A}_i^R$. Then, following the proof of Lemma 3.5 in \cite{Haussmann}, we arrive at
\begin{align*}
\int_{t_1}^{t_2}\int_U\phi^i(t,\psi_t,R^X(t),u)q_t(\d u)\d t\geq\limsup_{\ell\to\infty}\int_{t_1}^{t_2}\int_U\phi^i(t,\psi^{\ell}_t,R^X(t),u)q^{\ell}_t(\d u)\d t\geq 0,   \end{align*}
where we utilized the fact that $\phi^i$ is u.s.c. in $(x,\mu,u)\in\R^n\times\Pc_2(\R^n)\times\R^l$ imposed in Assumption~\ref{ass1}-(A2). Hence, $\omega\in\tilde{A}_i^R$, which yields the closedness of $\tilde{A}_i^R$.
				
Let $R_{\ell}\to R$ in $\Pc_2(\Omega;\widehat{\Omega})$ as $\ell\to\infty$ with $R_{\ell}$ satisfying $R_{\ell}(\tilde{A}_i^{R_{\ell}})=1$ for all $\ell\geq1$. Hence, $R_{\ell}(\tilde{A}_i^{R_{\ell}})-R_{\ell}(\tilde{A}_i^{R})=R_{\ell}(\tilde{A}_i^{R_{\ell}}\backslash\tilde{A}_i^R)$. However, we have
\begin{align*}
\tilde{A}_i^{R_{\ell}}\backslash\tilde{A}_i^R&\subset\left\{
\omega\in\Omega;~\int_{t_1}^{t_2}\int_U\phi^i(t,X_t(\omega),R_{\ell}^X(t),u)q_t(\omega,\d u)\d t\geq 0,~\text{and}\right.\\
&\qquad\int_{t_1}^{t_2}\left.\int_U\phi^i(t,X_t(\omega),R^X(t),u)q_t(\omega,\d u)\d t<0, ~\exists 0\leq t_1<t_2\leq T\right\}:= B_{\ell}.
\end{align*}
Note that $\phi^i$ is uniformly u.s.c. in $\mu$ w.r.t.  $(t,x,u)\in[0,T]\times\R^n\times\R^l$, which implies that
\begin{align*}
\phi^i(t,X_t(\omega),R^X(t),u)\geq \limsup_{\ell\to\infty}\phi^i(t,X_t(\omega),R_{\ell}^X(t),u)
\end{align*}
uniformly w.r.t. $(t,\omega,u)\in[0,T]\times\Omega\times\R^l$, and hence
\begin{align*}
\int_{t_1}^{t_2}\int_U\phi^i(t,X_t(\omega),R^X(t),u)q_t(\omega,\d u)\d t\geq\limsup_{\ell\to\infty}\int_{t_1}^{t_2}\int_U\phi^i(t,X_t(\omega),R_{\ell}^X(t),u)q_t(\omega,\d u)\d t
\end{align*}
uniformly w.r.t. $\omega\in\Omega$. Hence, $B_{\ell}$ is an empty set whenever $\ell$ is sufficiently large. Thus, it holds that
\begin{align*}
\limsup_{\ell\to\infty}R_{\ell}(\tilde{A}_i^{R_{\ell}})&\leq\limsup_{\ell\to\infty}\left(R_{\ell}(\tilde{A}_i^{R_{\ell}})-R_{\ell}(\tilde{A}_i^{R})\right)+\limsup_{\ell\to\infty}R_{\ell}(\tilde{A}_i^{R})\\
&=\limsup_{\ell\to\infty}R_{\ell}(\tilde{A}_i^{R_{\ell}}\backslash\tilde{A}_i^{R})+\limsup_{\ell\to\infty}R_{\ell}(\tilde{A}_i^{R})\leq R(\tilde{A}_i^R).
\end{align*}
For the last inequality, we used the fact that $\tilde{A}_i^R$ is closed and $R_{\ell}$ converges weakly to $R$ as $\ell\to\infty$. This yields that $R(\tilde{A}_i^R)=1$, and hence $(m\times R)(\tilde{A}_i^R)=T$, which completes the proof.
\end{proof}

\end{appendix}
		
		
		


		
\end{document}